\documentclass[preprint,12pt]{elsarticle}

\usepackage{lineno,hyperref}
\modulolinenumbers[5]

\usepackage{amssymb}
\usepackage{amscd}
\usepackage{amsthm}
\usepackage{amsxtra}
\usepackage{amsmath}
\usepackage{latexsym}
\usepackage{graphicx}
\usepackage{delarray}
\usepackage{float}
\usepackage[dvips]{epsfig}
\usepackage{subfig}
\usepackage{dsfont}
\usepackage{hyperref}
\usepackage{breakurl}

\begin{document}


\theoremstyle{plain}
\newtheorem{theorem}{Theorem}[section]
\newtheorem{claim}{Claim}
\newtheorem{lemma}[theorem]{Lemma}
\newtheorem{proposition}[theorem]{Proposition}
\newtheorem{corollary}[theorem]{Corollary}
\newtheorem{remark}{Remark}

\theoremstyle{remark}
\newtheorem{example}[theorem]{Example} 	
\theoremstyle{definition}
\newtheorem{definition}{Definition}
\hfuzz5pt 


\newcommand{\gt}{\tilde{g}}
\newcommand{\R}{\mathbb{R}}
\newcommand{\Z}{\mathbb{Z}}
\newcommand{\T}{\mathbb{T}}
\newcommand{\N}{\mathbb{N}}
\newcommand{\Zt}{\mathbb{Z}^2}
\newcommand{\Zd}{\mathbb{Z}^d}
\newcommand{\Ztd}{0\mathbb{Z}^{2d}}
\newcommand{\Rt}{\R^2}
\newcommand{\Rtd}{\R^{2d}}
\newcommand{\Zp}{Z_p}
\newcommand{\al}{\alpha}
\newcommand{\be}{\beta}
\newcommand{\om}{\omega}
\newcommand{\Om}{\Omega}
\newcommand{\ga}{\gamma}
\newcommand{\Ga}{\Gamma}
\newcommand{\bga}{\boldsymbol\ga}
\newcommand{\bg}{\mathbf{g}}
\newcommand{\de}{\delta}
\newcommand{\la}{\lambda}
\newcommand{\La}{\Lambda}
\newcommand{\ala}{\la^\circ}
\newcommand{\aLa}{\La^\circ}
\newcommand{\nat}{\natural}
\newcommand{\G}{\mathcal{G}}
\newcommand{\bG}{\mathbf{G}}
\newcommand{\A}{\mathcal{A}}
\newcommand{\V}{\mathcal{V}}
\newcommand{\M}{\mathcal{M}}
\newcommand{\MV}{\mathcal{MV}}
\newcommand{\MP}{\mathcal{MP}}
\newcommand{\Sp}{\mathcal{S}}
\newcommand{\W}{\mathcal{W}}
\newcommand{\Hp}{\mathcal{H}}
\newcommand{\Mep}{M_\epsilon}
\newcommand{\Kep}{K_\epsilon}
\newcommand{\ka}{\kappa}
\newcommand{\cast}{\circledast}
\newcommand{\id}{\mbox{Id}}
\newcommand{\by}{\mathbf{Y}}
\newcommand{\bx}{\mathbf{X}}
\newcommand{\bz}{\mathbf{Z}}
\newcommand{\bn}{\mathbf{N}}
\newcommand{\bv}{\mathbf{v}}
\newcommand{\bu}{\mathbf{u}}
\newcommand{\bA}{\mathbf{A}}
\newcommand{\bC}{\mathbf{C}}
\newcommand{\bD}{\mathbf{D}}
\newcommand{\bh}{\mathbf{h}}
\newcommand{\bff}{\mathbf{f}}
\newcommand{\bH}{\mathbf{H}}
\newcommand{\bV}{\mathbf{V}}
\newcommand{\bX}{\mathbf{x}}
\newcommand{\byy}{\mathbf{y}}
\newcommand{\bZ}{\mathbf{z}}
\newcommand{\bt}{\mathbf{t}}
\newcommand{\bs}{\mathbf{\sigma}}
\newcommand{\bI}{\mathbf{I}}
\newcommand{\wpe}{w_{\psi_\epsilon}}

\newcommand{\lo}{{\ell^1}}
\newcommand{\Lo}{{L^1}}
\newcommand{\Lt}{{L^2}}
\newcommand{\SO}{{S_0}}
\newcommand{\SOc}{{S_{0,c}}}
\newcommand{\Lp}{{L^p}}
\newcommand{\Los}{{L^1_s}}
\newcommand{\WCl}{{W(C_0,\ell^1)}}
\newcommand{\lt}{{\ell^2}}

\newcommand{\conv}[2]{{#1}\,\ast\,{#2}}
\newcommand{\twc}[2]{{#1}\,\nat\,{#2}}
\newcommand{\mconv}[2]{{#1}\,\cast\,{#2}}
\newcommand{\set}[2]{\Big\{ \, #1 \, \Big| \, #2 \, \Big\}}
\newcommand{\inner}[2]{\langle #1,#2\rangle}
\newcommand{\innerbig}[2]{\big \langle #1,#2 \big \rangle}
\newcommand{\innerBig}[2]{\Big \langle #1,#2 \Big \rangle}
\newcommand{\dotp}[2]{ #1 \, \cdot \, #2}

\newcommand{\Zpd}{\Zp^d}
\newcommand{\I}{\mathcal{I}}
\newcommand{\J}{\mathcal{J}}
\newcommand{\Zq}{Z_q}
\newcommand{\Zqd}{Zq^d}
\newcommand{\Zak}{\mathcal{Z}_{a}}
\newcommand{\C}{\mathbb{C}}
\newcommand{\F}{\mathcal{F}}

\newcommand{\convL}[2]{{#1}\,\ast_L\,{#2}}
\newcommand{\abs}[1]{\lvert#1\rvert}
\newcommand{\absbig}[1]{\big\lvert#1\big\rvert}
\newcommand{\absBig}[1]{\Big\lvert#1\Big\rvert}
\newcommand{\scp}[1]{\langle#1\rangle}
\newcommand{\norm}[1]{\lVert#1\rVert}
\newcommand{\normmix}[1]{\lVert#1\rVert_{\ell^{2,1}}}
\newcommand{\normbig}[1]{\big\lVert#1\big\rVert}
\newcommand{\normBig}[1]{\Big\lVert#1\Big\rVert}

\newcommand{\Conv}{\mathop{\scalebox{3}{\raisebox{-0.2ex}{$\ast$}}}}
\newcommand{\ConvN}{\mathop{\scalebox{3.5}{\raisebox{-0.2ex}{$\ast$}}}}
\newcommand{\Convtext}{\mathop{\scalebox{2.5}{\raisebox{-0.2ex}{$\ast$}}}}

\oddsidemargin  -0.31in
 \evensidemargin -0.31in
 \topmargin 0truept
 \headheight 5truept
 \headsep 0truept
 \footskip 20truept
 \textheight 235truemm
 \textwidth 180truemm
 \columnsep 8truemm


\begin{frontmatter}

\title{Frames of translates for model sets}

\author{Ewa Matusiak}
\address{Department of Mathematics, University of Vienna, Austria}
\ead{ewa.matusiak@univie.ac.at}


\begin{abstract}
We study spanning properties of a family of functions translated along simple model sets. We characterize tight frame and dual frame generators for such irregular translates and we apply the results to Gabor systems. We use the connection between model sets and almost periodic functions and rely strongly on a Poisson summations formula for model sets to introduce the so-called bracket product, which then plays a crucial role in our approach. As a corollary to our main results we obtain a density statement for semi-regular Gabor frames.
\end{abstract}

\end{frontmatter}


\section{Introduction}\label{sec:intro}
Frames of translates are an important class of frames that have a special structure. Here, generating functions $g_k$, where $k\in\mathcal{K}$ and $\mathcal{K}$ is a countable index set, are shifted along a regular set (lattice) $\La$ to create the analyzing family of elements, $\{g_k (t-\la):k\in\mathcal{K},\la\in\La\}$. These frames are central in approximation, sampling, Gabor and wavelet theory, and were investigated in the context of general properties of shift invariant spaces by a number of authors, including \cite{BDR94,RS95,B00}. In higher dimensions, translation invariant systems were investigated for example in \cite{L02,HLW02} and, more recently, on locally compact abelian groups in \cite{JL16} . The techniques rely on the Fourier analysis of periodic functions and the translation invariance of a lattice by shifts of its elements. 

For general irregular sets of translates, that is when $\La$ is not a lattice but some discrete relatively separated set, it is difficult to provide any constructive results as the tools to deal with such sets are missing. Certain extensions into irregular frames of translates were undertaken, for example in \cite{AG01} and \cite{BCHM11}. More recently in \cite{grrost17} the authors study nonuniform sampling in shift invariant spaces and construct semi-regular Gabor frames with respect to the class of totally positive functions. Their results are Beurling type results, expressed by means of density of the sampling sets.

In this article, we are interested in constructive results, that is results in terms of explicit conditions on a family $\{g_k (t-\la):k\in\mathcal{K},\la\in\La\}$ to form a frame of irregular translates. A class of irregular sets $\La$, so called model sets, possesses enough structure to enable us to formulate such conditions. Model sets were investigated recently by Matei and Meyer in the context of sampling and interpolation sets for the space of square integrable functions with compactly supported Fourier transform \cite{MaMe10, Me12, GL14}.  

We utilize the connection between model sets and almost periodic functions and use harmonic analysis of the latter to develop explicit conditions on $\{g_k (t-\la):k\in\mathcal{K},\la\in\La\}$ to be a tight frame for the space of square integrable functions. We rely strongly on Poisson summation formula for model sets to introduce the so called bracket product for model sets, in analogy to the bracket product for lattices introduced in \cite{BDR94}. One way to circumvent a problem of verifying if a given collection is a frame for its linear span is to search for a pair of dual frames instead. This approach opens for the possibility to specify in advance which properties, or structure we want for the dual frame. Since the irregular translates do not commute with the frame operator, the canonical dual frame is costly to compute and is not a frame of translates anymore. Therefore by searching for a pair of frames, we can ask for both of them to be frames of translates, generated from a given window by regular or irregular shifts. These kind of results, to our knowledge, are the first one in that direction. 

Almost periodic functions were recently investigated in the connection with Gabor frames in \cite{PU01,G04,BFG17}. As the space of almost periodic functions is non-separable, it can not admit countable frames, and the problem arises in which sense frame-type inequalities are still possible for norm estimation in this space \cite{G04,PU01,BFG17}. In \cite{BFG17} the authors also provide Gabor frames for a suitable separable subspaces of the space of almost periodic functions. We, on the other hand, use almost periodic functions as a tool to develop existence results for irregular frames for the space of square integrable functions.

The article is organized as follows. In Section~\ref{sec:not} we establish some notations and definitions that we will use throughout the article. We also shortly present main facts from the theory of almost periodic functions, we introduce model sets and point out some connections between the two. In Section~\ref{sec:bracket} we develop a technical tool, the bracket product, that we use in Section~\ref{sec:main} to characterize tight frames and  dual frames. We apply these results to the study of Gabor systems in Section~\ref{sec:gabor}.


\section{Preliminaries}\label{sec:not}

Let $T_\la$ denote the translation operator, $T_\la f(t) := f(t-\la)$ and $M_\la$ a modulation operator defined as $M_\la f(t) := e^{2\pi i \la \cdot t} f(t)$, where $\la\in\R^m$. Then $\F T_\la = M_{-\la} \F$ and $\F M_{\be} = T_{\be} \F$, where we define a Fourier transform as
\begin{equation*}
\F f (\om) = \widehat{f}(\om) = \int_{\R^m} f(t) e^{-2\pi i t \cdot \om}\, dt\,.
\end{equation*}
Moreover, $M_{\be} T_{\la} = e^{2\pi i \be\cdot \la} T_{\la} M_{\be}$.
By $\check{f}$ we denote the inverse Fourier transform. We define the convolution as $(\conv{f}{g})(x) = \int_{\R^m} f(t)g(x-t)\, dt$ and the involution $g^*(t) = \overline{g(-t)}$. We use the standard notation $\norm{f}_2$ for the norm of $f\in\Lt(\R^m)$, and $\inner{f}{g}$ for the usual inner product of $f,g\in \Lt(\R^m)$. For $1\leq p,q \leq \infty$, the Wiener amalgam space $W(L^p,\ell^q)$ denotes the space of functions $f$ such that
\begin{equation*}
\norm{f}_{p,q} := \left [ \sum_{k\in\Z^m} \norm{f\mathds{1}_k}_p^q \right ]^{1/q} < \infty\,,
\end{equation*}
where, $\mathds{1}_k$ is the characteristic function of the cube $[0,1]^m +k$, $k\in\Z^m$. Different partitions of $\R^m$ give equivalent Wiener amalgam norms. The space $W(L^\infty,\ell^1)$, which is a subspace of $L^1(\R^m)$ is called the {\it Wiener's algebra}, and it contains all bounded functions with compact support. It is therefore a dense subspace of $\Lt(\R^m)$.

Let $C_0(\R^m)$ be the space of all continuous functions that vanish at infinity. Then the closed subspace of $W(L^\infty,\ell^1)$ consisting of continuous functions is denoted by $W(C_0,\lo)$. Continuity of elements in $W(L^{\infty},\lo)$ allows for pointwise evaluations, and we have (see Proposition~$11.1.4$ in \cite{gr01}): if $F\in W(C_0,\lo)$, then $F|\La \in \lo(\La)$, for $\La$ any discrete relatively separated set in $\R^m$ with the norm estimate
\begin{equation}\label{eq:norm_estimate}
\sum_{\la\in\La} \abs{F(\la)} \leq \text{rel}(\La) \norm{F}_{\infty,1}\,.
\end{equation}
A discrete subset $\La$ of $\R^m$ is relatively separated if
\begin{equation*}
\text{rel}(\La) := \text{sup} \{ \# \{\La \cap B(x,1)\}:x\in\R^m\} < \infty\,,
\end{equation*}
where $B(x,1)=[0,1]^m+x$. 

Given a non-zero function $g\in\Lt(\R^m)$, the {\it short-time Fourier transform} of $f\in\Lt(\R^m)$ with respect to the window $g$, is defined as
\begin{equation*}
\V_g f (x,\om) := \int_{\R^m} f(t) \overline{g(t-x)} e^{-2\pi i \om \cdot (t-x)}\, dt = \inner{f}{T_x M_\om g}\,. 
\end{equation*}
The space $M^1(\R^m)$, known as Feichtinger's algebra, is then defined as
\begin{equation*}
M^1(\R^m) := \{ f\in\Lt(\R^m): \norm{\V_f f}_1 < \infty \}\,.
\end{equation*}
$M^1(\R^m)$ contains the Schwartz space $\mathcal{S}(\R^m)$ and it is dense in $L^p(\R^m)$, $1\leq p <\infty$. The following property of $M^1(\R^m)$ will be needed later.
\begin{proposition}\label{prop:STFT}
If $f,g\in M^1(\R^m)$, then $\V_g f \in W(C_0,\lo)(\R^{2m})$.
\end{proposition}
For the proof and more properties of $M^1(\R^m)$ we refer the reader to e.g \cite{fezi98} or \cite{gr01}.  We also mention the following tensor-product property of the short-time Fourier transform: for $\psi_1,\phi_1\in M^1(\R^n)$ and $\psi_2,\phi_2\in M^1(\R^{m-n})$
\begin{equation*}
\V_{\phi_1\otimes\phi_2}(\psi_1\otimes \psi_2)= \V_{\phi_1}\psi_1 \otimes \V_{\phi_2}\psi_2 \,.
\end{equation*}

Let $\mathcal{K}$ a countable index set. We are going to consider families of functions $T_\la g_k$, where $k\in \mathcal{K}$ and $\la\in\La$.  A collection $\{ T_\la g_k:k\in\mathcal{K},\la\in\La\}\subset \Lt(\R^m)$ is a {\it frame} for $\Lt(\R^m)$ if there exist constants $A_g,B_g>0$ such that
\begin{equation*}
A_g \norm{f}_2^2 \leq \sum_{k\in\mathcal{K}} \sum_{\la\in\La} \abs{\inner{f}{T_\la g_k}}^2 \leq B_g \norm{f}_2^2 \quad \mbox{for every $f\in\Lt(\R^m)$}\,,
\end{equation*}
The constants $A_g$ and $B_g$ are called lower and upper frame bounds, respectively. If $A_g=B_g$ then the frame is called a {\it tight frame}, and if $A_g=B_g=1$, a {\it normalized tight frame}. If $\{T_\la g_k:k\in\mathcal{K},\la\in\La\}$ satisfies the right inequality in the above formula, then it is called a {\it Bessel sequence}. 

Given a frame $\{ T_\la g_k:k\in\mathcal{K}, \la\in\La\}$ of $\Lt(\R^m)$ with lower and upper frame bounds $A_g$ and $B_g$, respectively, the {\it frame operator} $S_g^\La$, defined by
\begin{equation}\label{eq:frame operator}
S_g^{\La} f = \sum_{k\in\mathcal{K}} \sum_{\la\in\La} \inner{f}{T_\la g_k} T_\la g_k\,,
\end{equation}
is a bounded, invertible and positive mapping of $\Lt(\R^m)$ onto itself. This, in turn, provides the frame decomposition
\begin{equation}\label{eq:frame decomposition}
f = \sum_{k\in\mathcal{K}} \sum_{\la\in\La}  \inner{f}{\big (S_g^{\La} \big)^{-1} T_\la g_k} T_\la g_k\,, \quad \mbox{for all $f\in\Lt(\R^m)$.}
\end{equation}
The sequence $\{\big(S_g^\La \big )^{-1} T_\la g_k:k\in\mathcal{K}, \la\in\La\}$ is also a frame for $\Lt(\R^m)$, called the {\it canonical dual frame} of $\{ T_\la g_k:k\in\mathcal{K},\la\in\La\}$, and has upper and lower frame bounds $B_g^{-1}$ and $A_g^{-1}$, respectively. If the frame is tight, then $\big (S_g^\La \big )^{-1}= A_g^{-1} I$, where $I$ is the identity operator, and the frame decomposition becomes
\begin{equation*}
f = A_g^{-1}\,\sum_{k\in\mathcal{K}} \sum_{\la\in\La} \inner{f}{T_\la g_k} T_\la g_k\,,\quad \mbox{for all $f\in\Lt(\R^m)$}\,.
\end{equation*}
In order to use the representation \eqref{eq:frame decomposition} in practice, we need to be able to compute $\big (S_g^\La \big )^{-1}$. While the existence of $\big (S_g^\La \big )^{-1}$ is guaranteed by the frame condition, it is usually tedious to find this operator explicitly. Moreover, if $\La$ is not a lattice in $\R^m$, then the frame operator $S_g^{\La}$ does not commute with translations along $\La$, because $\La$ is not closed under translations of its elements. Indeed, let $\alpha\in\La$ and $f\in\Lt(\R^m)$, then
\begin{align*}
T_{-\alpha} \,S_g^{\La} \,T_\alpha f &= \sum_{k\in\mathcal{K}} \sum_{\la\in\La} \inner{T_\alpha f}{T_\la g_k} T_{-\alpha}T_\la g_k= \sum_{k\in\mathcal{K}} \sum_{\la\in\La}  \inner{f}{T_{\la-\alpha} g_k} T_{\la-\alpha} g_k \\
&= \sum_{k\in\mathcal{K}} \sum_{\la\in\La - \alpha} \inner{f}{T_\la g_k}T_\la g_k = S_g^{\La - \alpha} f \,,
\end{align*}
where $S_g^{\La-\alpha}$ is the operator associated to $\{ T_\la g_k:k\in\mathcal{K}, \la\in\La-\alpha\}$. In fact a more general result holds,

\begin{proposition}\label{prop:operator shifts}
Let $\La$ be a relatively separated set in $\R^m$. If $\{T_\la g_k:k\in\mathcal{K}, \la\in\La\}$ is a frame (Bessel sequence) for $\Lt(\R^m)$, then 
\begin{equation*}
T_{-x} \,S_g^{\La} \,T_x = S_g^{\La-x}\,, \quad \mbox{for all $x\in\R^m$}\,,
\end{equation*}
and $\{T_\la g_k:k\in\mathcal{K}, \la\in\La-x\}$ is a frame (Bessel sequence) for $\Lt(\R^m)$, for all $x\in\R^m$, with the same upper and lower frame bounds..
\end{proposition}

By Proposition~\ref{prop:operator shifts}, the canonical dual frame $\{ \big (S_g^{\La}\big)^{-1} T_\la g_k:k\in\mathcal{K}, \la\in\La\}$ does not have the same structure as $\{T_\la g_k:k\in\mathcal{K}, \la\in\La\}$, that is it is not a frame of translates. Therefore, in order to compute the canonical dual frame we would have to apply $\big (S_g^{\La} \big)^{-1}$ to $T_\la g_k$, for all $k\in\mathcal{K}$ and all $\la\in\La$. To avoid the mentioned complications, we search for a pair of dual frames, rather than just one frame. Let $\{T_\la g_k:k\in\mathcal{K},\la\in\La\}$ and $\{T_\la h_k:k\in\mathcal{K}, \la\in\La \}$ be Bessel sequences for $\Lt(\R^m)$, then we can define a mixed frame operator
\begin{equation*}
S_{g,h}^{\La} f = \sum_{k\in\mathcal{K}} \sum_{\la\in\La} \inner{f}{T_\la g_k} T_\la h_k
\end{equation*}
which is a bounded linear operator on $\Lt(\R^m)$. If $S_{g,h}^{\La} f = f$ for every $f\in\Lt(\R^m)$, then we call $\{T_\la h_k:k\in\mathcal{K}, \la\in\La\}$ a  dual frame of $\{T_\la g_k:k\in\mathcal{K}, \la\in\La \}$. The covariance relationship of Proposition~\ref{prop:operator shifts} holds also for mixed frame operators.


\subsection{Almost Periodic Functions}\label{sec:ap}

In the treatment of frames of irregular translates, with translates originting from model sets, we naturally come across almost periodic functions. We will consider classes of almost periodic functions with spectrum in a model set and use some of their properties to form our results. We review some basic facts about almost periodic functions that we use later in the article. For a more detailed exposition we refer to \cite{B47, Bes55,BB31}.

We say that a bounded and continuous function $f:\R^m\rightarrow \C$ is {\it almost periodic}, if to every $\epsilon >0$ there corresponds a relatively dense set $E(f,\epsilon) \subseteq \R^m$, such that for every $\tau\in E(f,\epsilon)$,
\begin{equation*}
\sup_{t\in\R^m}\abs{f(t + \tau) - f(t)} \leq \epsilon\,.
\end{equation*}
A subset $D$ is called relatively dense in $\R^m$ when there exists $r>0$, such that for all $t\in\R^m$, $D\cap B(t,r) \neq \emptyset$, where $B(t,r)=r[0,1]^m + t$. Each $\tau\in E(f,\epsilon)$ is called an $\epsilon$-period of $f$. Let $AP(\R^m)$ denote the space of almost periodic functions. Each almost periodic function is uniformly continuous.

\noindent For every almost periodic function the mean value
\begin{equation*}
\M\{f\} = \M_t\{f(t)\} = \lim_{R\rightarrow \infty} \frac{1}{R^m} \int_{B(x,R)} f(t) \,dt\,,
\end{equation*}
where $B(x,R)=R[0,1]^m+x$, $x\in\R^M$ exists and is independent of $x\in\R^m$. The following also holds
\begin{theorem}\label{thm:ap in 2 variables}
If $f(x,t)$, with $(x,t) \in \R^m \times \R^m$, is almost periodic, then it is almost periodic with respect to each of the variables $x$ and $t$. Moreover, $\M_t\{ f(x,t)\}$ is an almost periodic function of $x$.
\end{theorem}
Since the mean value of every almost periodic function exists, we have the harmonic analysis of almost periodic functions. For each $\la\in\R^m$,
\begin{equation*}
a(\la,f) := \M_t\{f(t) e^{2\pi i \la \cdot t}\}\,,
\end{equation*} 
are Fourier coefficients of $f$, and are nonzero only for a countable number of $\la\in\R^m$ (\cite{B47}). The values $\la\in\R^m$ for which $a(\la,f)\neq 0$ are called the characteristic exponents of $f$ and they constitute the co-called Bohr spectrum of $f$,
\begin{equation*}
\sigma(f) := \big\{\lambda\in\R^m \,:\, a(\la,f) \neq 0 \big \}\,.
\end{equation*}
To every $f\in AP(\R^m)$ we can associate the formal Fourier series
\begin{equation}\label{eq:Fourier series}
f(t) \sim \sum_{\la \in \sigma(f)} a(\la,f) e^{-2\pi i \la \cdot  t}\,,
\end{equation}
and two almost periodic functions $f$ and $g$ are equal if and only if their Fourier coefficients equal $a(\la,f)=a(\la,g)$, for all $\la\in\R^m$. This is the Uniqueness Theorem (\cite{B47}) for almost periodic functions. 

If $\{f_n\}_{n\in\N}$ is a sequence of almost periodic functions, which converges uniformly to the limit function $f$, which is by the way also almost periodic, then $\M\{f\} = \lim_{n\rightarrow \infty} \M\{f_n\}$. Moreover, if 
\begin{equation*}
f_n(t) \sim \sum_{\la_n\in \sigma(f_n)} a(\la_n,f_n)e^{-2\pi i \la_n \cdot  t}\,,
\end{equation*}
then 
\begin{equation*}
\M_t\{f(t) e^{2\pi i \la \cdot  t}\} = \lim_{n\rightarrow \infty} \M_t\{ f_n(t) e^{2\pi i \la \cdot  t}\}
\end{equation*}
holds uniformly for all $\la\in\R^m$.

We conclude this section with three important results that we will need later. The following form of Parseval's equation holds for almost periodic functions.
\begin{theorem}[Bohr's Fundamental Theorem]\cite{B47}\label{thm:Bohr}
Let $f\in AP(\R^m)$ with Fourier series given by \eqref{eq:Fourier series}. Then
\begin{equation*}
\M_t\{\abs{f(t)}^2\} = \sum_{\la\in\sigma(f)} \abs{a(\la,f)}^2\,.
\end{equation*}
\end{theorem}
It follows directly from the Uniqueness Theorem and Theorem~\ref{thm:Bohr}, that if all the coefficients $a(\la,f)$ of $f\in AP(\R^m)$ are zero, then the function $f\equiv 0$. Moreover, for non-negative almost periodic functions, we have
\begin{theorem}\label{thm:ap non-negative}
Let $f\in AP(\R^m)$ be non-negative. Them $\M \{ f \} = 0$ if and only if $f\equiv 0$.
\end{theorem}
Note that the mean value of an almost periodic function is its $0-$th Fourier coefficient. Analogously to the case of periodic functions, there exists Plancherel's Theorem  for almost periodic functions
\begin{theorem}[Plancherel's Theorem]\label{thm:Plancherel}
Let $f,g \in AP(\R^m)$ and such that $\sigma(f) = \sigma(g)$. Then
\begin{equation*}
\M_t\{ f(t)\, \overline{g(t)} \} = \sum_{\la\in\sigma(f)} a(\la,f)\,\overline{a(\la,g)} \,.
\end{equation*}
\end{theorem}


\subsection{Model Sets}\label{sec:ms}

Model sets were introduced by Meyer \cite{Me72} in his study of harmonious sets. We begin with a lattice in $\Ga \subset \R^m \times \R^n$, where $\R^m$ and $\R^n$ are equipped with Euclidean metrics and $\R^m \times \R^n$ is the orthogonal sum of the two spaces. Let $p_1:\R^m \times \R^n \rightarrow \R^m$ and $p_2: \R^m \times \R^n \rightarrow \R^n$ be orthogonal projection maps such that $p_1|\Ga$ is injective and $L= p_1(\Ga)$ is a dense subgroup of $\R^m$. We impose the same properties on $p_2$. 
For the dual lattice of $\Ga$, denoted by $\Ga^\ast$, let $p_1^*,p_2^*$ be defined as $p_1,p_2$. It holds then, that  $p_1^*|\Ga^\ast$ is injective and $p_1^*(\Ga^\ast)$ is a dense subgroup of $\R^m$, and the same holds for $p_2^*$. 
Moreover, for $\ga\in\Ga$ and $\ga^*\in\Ga^*$,
\begin{equation*}
\Z \ni \ga \cdot \ga^* = (p_1(\ga),p_2(\ga))\cdot (p_1^*(\ga^*),p_2^*(\ga^*)) = p_1(\ga) \cdot p_1^*(\ga^*) + p_2(\ga)\cdot p_2^*(\ga^*)
\end{equation*}
Let $\Om\subset \R^n$ be compact, equal to the closure of its interior, and to have boundary of Lebesgue measure $0$. We call $\Om$ a window. Then the model set $\La(\Om)$ is defined as 
\begin{equation*}
\La(\Om) :=  \left \{p_1(\ga)\,:\, \ga\in\Ga,\, p_2(\ga) \in \Om \right \}  \subset L \subset \R^m\,.
\end{equation*}
Model set is {\it simple} if $n=1$ and $\Om$ is an interval, and it is {\it generic} if the boundary of $\Om$ has no common points with $p_2(\Ga)$. Additionally, if $\Om$ is symmetric around the origin then $0\in\La(\Om)$.  We will be working only with simple model sets. We assume, without loss of generality, that from now on $\Om$ is symmetric around the origin.

\begin{example}
Let $\Ga \subseteq \R^m \times \R$ be a lattice given by
\begin{equation*}
\Ga = \{(\mathbf{I}+\be \al^T)k - \be l,l-\al^T k): k\in\Z^m\,, l\in\Z\}\,,
\end{equation*}
where $\mathbf{I}$ is the $m\times m$ identity matrix, $\al=(\al_1,\ldots,\al_m)^T$ and $\be=(\be_1,\ldots,\be_m)^T$ are column vectors in $\R^m$ such that the numbers $1,\al_1,\ldots,\al_m$ are linearly independent over the rationals, and the numbers $\be_1,\ldots,\be_m,1+\be^T \al$ are linearly independent over the rationals. Then for $\Om=[-1,1]$, $\La(\Om)$ is a simple model set. 
\end{example}

Model sets are a very natural generalizations of lattices, and for $n=0$ they reduce to a lattice and, thus, the results that we develop later on in the article reduce to the known ones for lattices. If $\La(\Om)$ is a model set, then it is uniformly discrete (e.g. there is an open ball $B(0,r)$ such that $(\La(\Om) - \La(\Om))\cap B(0,r) = \{0\}$), relatively dense, and has a well defined density 
\begin{equation*}
D(\La(\Om)) = \lim_{R\rightarrow \infty} \frac{\# \{\La(\Om)\cap B(x,R)\}}{R^m}\,,
\end{equation*}
where $\#S$ denotes the cardinality of the set $S$ and $B(x,R) = R[0,1]^m+x$. The limit is independent of $x\in\R^m$. When  $\La(\Om)$ is a simple model set it can be shown that $D(\La(\Om)) = \text{vol}(\Ga)^{-1} \abs{\Om}$, \cite{Sch98}.

Due to the underlying lattice structure of a model set, there exists a Poisson summation formula for $\La(\Om)$. Let $C_0^\infty(\Om)$ be the space of all smooth, real valued functions on $\R$ with support in $\Om$. Via the mapping $p_2\circ(p_1|_\Ga)^{-1}: L \rightarrow \R$ we obtain a space $\mathcal{C}(\La(\Om))$ of functions on $L$, vanishing off $\La(\Om)$:  for $\psi \in C_0^\infty(\Om)$, we define $w_\psi \in \mathcal{C}(\La(\Om))$ by
\begin{equation}\label{def:w_psi}
w_\psi: L \rightarrow \R \,,\quad \quad w_\psi(\la) := \psi(p_2(\ga)) \quad \mbox{for} \quad \la=p_1(\ga)\in \La(\Om)\,,
\end{equation}
and $w_\psi(\la)=0$ for $\la\notin \La(\Om)$. If $\psi$ were the indicator function of $\Om$, we would have $w_\psi(\la) = 1$ on $\La(\Om)$ and $w_\psi(\la)=0$ if $\la \notin \La(\Om)$. However, the indicator function is not smooth. The Poisson summation formula for model sets was originally stated for the class of Schwartz functions in \cite{Me12}. However, since it relies on the original Poisson summation formula, we can state it for a bigger space.

\begin{theorem}[Poisson Summation Formula]\label{thm:poisson}
Let $\La(\Om)$ be a model set defined by a relatively compact set $\Om\subseteq \R$ of non-empty interior and a lattice $\Ga\subseteq \R^m \times \R$. Let $\psi\in C_0^\infty(\Om)$, and the weight factors $w_\psi(\la)$ on $\La(\Om)$ be defined as in \eqref{def:w_psi}. Then, for every $F\in W(L^\infty,\ell^1)$ such that $\widehat{F}\in W(L^\infty,\ell^1)$, the following holds
\begin{equation}\label{eq:poisson}
\sum_{\la\in\La(\Om)} w_\psi(\la) F(\la)e^{-2\pi i \la \cdot t} =  \sum_{\ga^\ast \in \Ga^\ast} \widetilde{w}_\psi(-p^*_2(\ga^\ast))  \widehat{F}(t - p^*_1(\ga^\ast))\,,
\end{equation}
where 
\begin{equation}\label{eq:Fourier weight}
\widetilde{w}_\psi(p^*_2(\ga^\ast)) := \text{vol}(\Ga)^{-1} \widehat{\psi}(p^*_2(\ga^\ast))\quad \quad \mbox{for} \quad \ga^\ast\in \Ga^\ast\,.
\end{equation}
The identity holds pointwise for all $t\in\R^m$, and both sums converge uniformly and absolutely for all $t\in\R^m$. \end{theorem}

Since $w_\psi(\la) = \psi(p_2(\ga))$ for $\la=p_1(\ga)$, one can forget about the restriction $\la\in\La$ which is given for free by the support of $\psi$. Then \eqref{eq:poisson} follows from the ordinary Poisson summation formula applied to the lattice $\Ga^\ast$ and its dual lattice $(\Ga^\ast)^\ast = \Ga$. 

\begin{remark}
Meyer, in \cite{Me12}, stated the Poisson summation formula for model sets for functions in the Schwartz class $\mathcal{S}(\R^m)$. Ordinary Poisson summation formula holds pointwise on a larger space, namely for functions in Wiener algebra $W(L^\infty,\ell^1)$ whose Fourier transforms are also in Wiener algebra $W(L^\infty,\ell^1)$.  If both $F$ and $\widehat{F}$ are in $W(L^\infty,\ell^1)$, then $F\otimes \psi$ and $\widehat{F}\otimes\widehat{\psi}$ are in  $W(L^\infty,\ell^1)(\R^m\times \R)$, they are continuous and the series 
\begin{equation*}
\sum_{\ga\in\Ga} \psi(p_2(\ga)) F(p_1(\ga)) e^{-2\pi i (p_1(\ga),p_2(\ga)) \cdot (t,x)}\quad \mbox{and} \quad \text{vol}(\Ga)^{-1} \sum_{\ga^\ast \in \Ga^\ast} \widehat{\psi}(x- p^*_2(\ga^\ast))  \widehat{F}(t - p^*_1(\ga^\ast)
\end{equation*}
converge absolutely by the very definition of $W(L^\infty,\ell^1)$. These sums are equal for all $(t,x)\in\R^m\times \R$, and taking $x=0$ yields \eqref{eq:poisson}. This is in some sense the largest "natural" space on which the Poisson summation formula holds pointwise. Therefore, we could extend the Poisson summation formula for model sets to hold for $F$ such that $F,\widehat{F}\in W(L^\infty,\ell^1)$. 

Another natural space for which the Poisson summation formula for model sets holds is the Feichtinger's algebra $M^1$. It follows from the fact, that if $F\in M^1$, then $F\in W(L^\infty,\ell^1)$ and $\widehat{F} \in W(L^\infty,\ell^1)$.
\end{remark}

\begin{remark}\label{rem:a.p. construction}
Poisson summation formula for model sets gives a method for constructing almost periodic functions with desired spectrum. Indeed, the function on the right hand side of \eqref{eq:poisson} is almost periodic since it equals an absolutely convergent trigonometric series. By the property of almost periodic functions, the Fourier series of this function coincides with this trigonometric series. That means that the Fourier coefficients of the right hand side of \eqref{eq:poisson} equal $w_\psi(\la)F(\la)$.
\end{remark}

On the collection of point sets in $\R^m$ that are relatively dense and uniformly separated, with minimal separation greater than $r$, denoted by $D_r(\R^m)$, we can put a topology, called local topology: two sets $\La$ and $\La'$ of $D_r(\R^m)$ are close if, for some large $R$ and some small $\epsilon$, one has
\begin{equation}\label{eq:topology}
\La'\cap B(0,R) = (\La +v)\cap B(0,R) \quad \text{for some $v\in B(0,\epsilon)$.}
\end{equation}
Thus for each point of $\La$ within the ball $B(0,R)$, there is a point of $\La'$ within the distance $\epsilon$ of that point, and vice versa. The pairs $(\La,\La')$ satisfying \eqref{eq:topology} are called $(R,\epsilon)$-close. More formally, for $\epsilon >0$ and a ball $B(x,R)$, define
\begin{equation*}
U(\epsilon,B(x,R)):= \{(\La,\La')\in D_r(\R^m)\times D_r(\R^m):(\La +v)\cap B(x,R) = \La '\cap B(x,R), \mbox{for some}\, v\in B(0,\epsilon)\}\,.
\end{equation*}
These sets form a fundamental system for a uniform structure on $D_r(\R^m)$ whose topology has the sets
\begin{equation*}
U(\epsilon,B(x,R))[\La] := \{ \La' \in D_r(\R^m):(\La,\La')\in U(\epsilon,B(x,R))\}
\end{equation*}
as a neighbourhood basis of $\La$. Note, all the point sets $\La$ from $D_r(\R^m)$ have the same relative separation $\text{rel}(\La)$.

On the set $D_r(\R^m)$ we can put a metric. Let $\La,\La' \in D_r(\R^m)$, then
\begin{equation*}
d(\La,\La') := \limsup_{R\rightarrow \infty} \frac{\# \{ ((\La \cup \La')\setminus (\La \cap \La')) \cap B(0,R)\}}{R^{m}}
\end{equation*}
is a pseudometric on $D_r(\R^m)$. We obtain a metric by defining the equivalence relation
\begin{equation*}
\La \equiv \La' \iff d(\La,\La')=0\,.
\end{equation*}

Later in the article, we will work with a collection of model sets. Let $\Om$ be a window, then for each $(t,s)\in\R^m\times\R^n$ we may define
\begin{equation*}
\La_{(t,s)}^\Om = t + \La(\Om - s)
\end{equation*}
Note that $\La(\Om)$ and all its shifts have the same relative separation $\text{rel}(\La(\Om))$.

If $(t,s)\equiv(t',s') \mod \Ga$, then $\La_{(t,s)}^\Om=\La_{(t',s')}^\Om$, however the inverse is not necessarily true. In the sequel we will write $(t,s)_L$ for the congruence class $(t,s)\mod \Ga$. These model sets are parametrized  by the torus  $\T^{m+n}:=(\R^m\times \R^n) / \Ga = (\R/\Z)^{m+n}$. There is  a natural measure, Haar measure, $\theta$ on $\T^{m+n}$. It is invariant under the action of $\R^m$ on $(\R^m\times \R^n )/\Ga$ and it acts by
\begin{equation*}
x+(t,s)_L = (x+t,s)_L\,.
\end{equation*}
We can define an embedding $\R^m \rightarrow \T^{m+n}$, $x\mapsto (x,0)_L$. The image of this embedding is dense in $\T^{m+n}$.

Now, let $\La(\Om)$ be a model set, and we translate it by elements $x\in\R^m$
\begin{equation*}
x+\La(\Om) = x + \La(\Om+0) = \La_{(x,0)}^\Om\,.
\end{equation*}
The closure of the set of all translates $\La_{(x,0)}^\Om$ of $\La(\Om)$ under the local topology \eqref{eq:topology} forms the so-called local hull $X(\La(\Om))$ of $\La(\Om)$, $X(\La(\Om)) = \overline{\{x+\La(\Om):x\in\R^m\}}$, (\cite{M05},\cite{Sch00}).

\begin{proposition}\cite{Sch00}
Let $\La(\Om)$ be a model set. There is a continuous mapping 
\begin{equation*}
\beta :\, X(\La(\Om)) \rightarrow \T^{m+n}\,,
\end{equation*}
called the torus parametrization, such that $1)$ $\beta$ is onto; $2)$ $\beta$ is injective almost everywhere with respect to the Haar measure $\theta$; $3)$ $\beta(x+\La') = x+ \beta(\La')$ for all $x\in\R^m$ and all $\La'\in X(\La)$; and $4)$ $\beta(x+\La(\Om)) = (x,0)_L$ for all $x\in\R^m$.
\end{proposition}
By injective almost everywhere, we mean that the set $P$ of points $x\in\T^{m+n}$, for which there is more than one point set of $X(\la(\Om))$ over $x$, satisfies $\theta(P)=0$.

There is a unique $\R^m$-invariant measure $\mu$ on $X(\La(\Om))$, with $\mu(X(\La(\Om))) = 1$, and $\beta$ relates the Haar measure $\theta$ and $\mu$ through: $\theta(P) = \mu(\beta^{-1} P)$ for all measurable subsets $P$ of $\T^{m+n}$. Having $\mu$ we can introduce the space $\Lt(X(\La(\Om)),\mu)$ of square integrable functions on $X(\La(\Om))$. Square integrable functions on $X(\La(\Om))$ and square integrable functions on $\T^{m+n}$ can be identified,
\begin{equation}\label{eq:iso}
\Lt(X(\La(\Om)),\mu) \simeq \Lt(\T^{m+n},\theta)\,.
\end{equation}
The mapping takes a function $\mathfrak{N}\in \Lt(\T^{m+n},\theta)$ and creates $\widetilde{\mathcal{N}}=\mathfrak{N} \circ \beta \in \Lt(X(\La(\Om)),\mu)$, and since $\beta$ is almost everywhere injective, the map is a bijection. This will allow us later to analyze functions on $X(\La(\Om))$ by treating them as functions on $\T^{m+n}$.


\subsection{Local functions on model sets}\label{sec:local functions}

Let $\La(\Om)$ be a model set in $\R^m$ with window $\Om$, arising from the lattice $\Ga \subset \R^m\times \R^n$. Let $\T^{m+n}=(\R^m\times \R^n)/\Ga$ be the torus with torus parametrization $\beta: X(\La(\Om)) \rightarrow \T^{m+n}$. Then we have the identification \eqref{eq:iso} of the corresponding $\Lt$ spaces. Each element $\La'\in X(\La(\Om))$ maps by $\beta$ to a point $\beta(\La')$ in $\T^{m+n}$. We also know $\beta(\La(\Om)) = (0,0)_L$, and $\beta(\La(\Om) + x) = (x,0)_L$. So we know how $\beta$ works on $\R^m + \La(\Om)$. 

Consider a function $\widetilde{\mathcal{N}}:X(\La(\Om)) \rightarrow \C$. We can define from it a function $\mathcal{N}:\R^m \rightarrow \C$ by 
\begin{equation*}
\mathcal{N}(x) = \widetilde{\mathcal{N}}(x+\La(\Om))\,.
\end{equation*}
If $\widetilde{\mathcal{N}}$ is continuous, then for all $x_1,x_2\in\R^m$, if $x_1+\La(\Om)$ and $x_2+\La(\Om)$ are close, then $\widetilde{\mathcal{N}}(x_1+\La(\Om))$ and $\widetilde{\mathcal{N}}(x_2+\La(\Om))$ are close, and as a consequence, $\mathcal{N}(x_1)$ and $\mathcal{N}(x_2)$ are close. Thus continuity of $\widetilde{\mathcal{N}}$ implies continuity of $\mathcal{N}$, or a certain locality. More formally, a function $\mathcal{N}:\R^m\rightarrow \C$ is called {\it local with respect to} $\La(\Om)$, if for all $\delta>0$ there exists $R>0$ and $\epsilon>0$ so that whenever $x_1+\La(\Om)$ and $x_2+\La(\Om)$, for $x_1,x_2\in\R^m$, are $(R,\epsilon)$-close, then
\begin{equation*}
\abs{\mathcal{N}(x_1)-\mathcal{N}(x_2)} < \delta\,.
\end{equation*}  
Intuitively, $\mathcal{N}$ looks very much the same at places where the local environment looks the same. It can be easily verified that local functions are continuous on $\R^m$ and almost periodic.

Using locality, we can go in the opposite direction. Let $\mathcal{N}$ be a local function with respect to $\La(\Om)$. Define a function $\widetilde{\mathcal{N}}$ on the orbit of $\La(\Om)$:
\begin{equation*}
\widetilde{\mathcal{N}}:\{ x+\La(\Om):x\in\R^m\} \rightarrow \C \quad \mbox{by} \quad \widetilde{\mathcal{N}}(x+\La(\Om)) = \mathcal{N}(x)\,.
\end{equation*}
Then $\widetilde{\mathcal{N}}$ is uniformly continuous on $\{ x+\La(\Om):x\in\R^m\}$ with respect to the local topology. The reason for this is that the continuity condition which defines the localness of $\mathcal{N}$ is based on the uniformity defining the local topology on $\{ x+\La(\Om):x\in\R^m\}$. It follows that $\widetilde{\mathcal{N}}$ lifts uniquely to a continuous function on a local hull $X(\La(\Om))$.

\begin{proposition}\cite{MNP08}
For each local function $\mathcal{N}$ with respect to $\La(\Om)$ there is a unique continuous function $\widetilde{\mathcal{N}}$ on a local hull $X(\La(\Om))$, whose restriction to the orbit of $\La(\Om)$ is $\mathcal{N}$. Every continuous function on the local hull of $\La(\Om)$ arises in this way.
\end{proposition}

The spectral theory of $\Lt(X(\La(\Om)),\mu)$ allows us to analyze $\mathcal{N}$ by analyzing its corresponding function $\widetilde{\mathcal{N}}$ on $\Lt(X(\La(\Om)),\mu)$. Suppose $\mathcal{N}$ is a local function with respect to the model set $\La(\Om)$. From the locality of $\mathcal{N}$ we have its extension $\widetilde{\mathcal{N}}\in\Lt(X(\La(\Om)),\mu)$ which is continuous. Then we obtain $\mathfrak{N}\in \Lt(\T^{m+n},\theta)$, where
\begin{equation*}
\mathfrak{N}((x,0)_L) = \mathfrak{N}(\beta(x+\La(\Om))) = \widetilde{\mathcal{N}}(x+\La(\Om)) = \mathcal{N}(x)\,,
\end{equation*}
and since functions in $\Lt(\T^{m+n},\theta)$ have Fourier expansions, we can write
\begin{equation}\label{eq:Fourier series N}
\mathcal{N}(x) = \widetilde{\mathcal{N}}(x+\La(\Om)) = \mathfrak{N}((x,0)_L) = \sum_{\eta\in\Ga^*} \widehat{\mathfrak{N}}(\eta) e^{-2\pi i (x,0) \cdot \eta} = \sum_{\eta\in\Ga^*} \widehat{\mathfrak{N}}(\eta) e^{-2\pi i x\cdot p_1^*(\eta)}\,,
\end{equation}
almost everywhere, with 
\begin{equation*}
\widehat{\mathfrak{N}}(\eta) = \int_{\T^{m+n}} \mathfrak{N}((t,s)_L) e^{2\pi i (t,s) \cdot \eta}\, d\theta(t,s)\,.
\end{equation*}
However, we know $\mathfrak{N}$ only on $(\R^m,0)_L$. To compute the coefficients $\widehat{\mathfrak{N}}(\eta)$ out of $\mathcal{N}$ alone, we can use the Birkhoff ergodic theorem
\begin{align*}
\widehat{\mathfrak{N}}(\eta)  &= \int_{\T^{m+n}} \mathfrak{N}((t,s)_L) e^{2\pi i (t,s) \cdot \eta}\, d\theta(t,s)= \lim_{R\rightarrow \infty} \frac{1}{R^m} \int_{B(0,R)} \mathfrak{N}((x,0)_L) e^{2\pi i (x,0) \cdot \eta}\, dx\\
&= \lim_{R\rightarrow \infty} \frac{1}{R^m} \int_{B(0,R)} \mathcal{N}(x) e^{2\pi i x \cdot p_1^*(\eta)}\, dx\,,
\end{align*}
where we used $\mathfrak{N}((x,0)_L) = \mathcal{N}(x)$ and $\eta = (p_1^*(\eta),p_2^*(\eta))$, so
\begin{equation*}
(x,0)\cdot \eta = x \cdot p_1^*(\eta) + 0 \cdot p_2^*(\eta)\,.
\end{equation*}
If $ \sum_{\eta\in\Ga^*} \abs{\widehat{\mathfrak{N}}(\eta)}<\infty$, then the Fourier series \eqref{eq:Fourier series N} converges absolutely to $\mathcal{N}(x)$ for all $x\in\R^m$.


\section{Bracket product}\label{sec:bracket}

As described in the introduction, we are interested in the spanning properties of a family of functions shifted along a simple model set. In order to characterize properties of the families $\{T_\la g_k:k\in\mathcal{K}, \la\in\La(\Om)\}$, in particular tight frame property and  dual frames, we need to develop some useful techniques. We introduce the so called bracket product, in a similar manner as it is done in the case of regular shifts along a lattice. Bracket product was initially introduced in \cite{BDR94}, and generalized to other lattices in \cite{L02},\cite{HLW02}.

We assume from now on that $\Om$ is symmetric around the origin and that $p_2(\Ga)$ and $p_2^*(\Ga^\ast)$ have no common points with the boundary of $\Om$. Let $\La(\Om)$ be a simple model set and $\psi\in C_0^\infty(\Om)$. Let $\widetilde{w}_\psi$ be a function defined as in Theorem~\ref{thm:poisson}. Then the $\psi${\it -bracket product} of $f$ and $g$ is defined as
\begin{equation}\label{def:bracket}
\big [\widehat{f},\widehat{g}\big ]^{\psi} (t) := \sum_{\ga^\ast \in \Ga^\ast} \widetilde{w}_\psi(-p^*_2(\ga^\ast)) \widehat{f}(t - p^*_1(\ga^\ast)) \overline{\widehat{g}(t-p^*_1(\ga^\ast))}.
\end{equation}
It can be easily seen, that for $f,g\in W(\Lt,\ell^1)$, we have $\widehat{f} \overline{\widehat{g}} \in W(L^\infty,\ell^1)$ and the bracket product is well defined. Moreover, $f\ast g^* \in W(L^\infty,\ell^1)$, and by Remark~\ref{rem:a.p. construction}, $\big [ \widehat{f},\widehat{g} \big ]^{\psi}$ is an almost periodic function represented by the trigonometric series
\begin{equation*}
\big [ \widehat{f},\widehat{g} \big ]^{\psi}(t) = \sum_{\la\in\La} w_\psi(\la) \inner{f}{T_\la g} e^{-2\pi i \la \cdot t}\,,
\end{equation*}
since $(\conv{f}{g^*})(\la)=\inner{f}{T_\la g}$, and the Fourier coefficients are given by
\begin{equation}\label{eq:Fourier coeff}
\M_t\Big \{\big [ \widehat{f},\widehat{g} \big ]^{\psi}(t) e^{2\pi i \la \cdot t} \Big \} = w_{\psi}(\la) \inner{f}{T_\la g}\,.
\end{equation}
It immediately follows that $\inner{f}{T_\la g} =0$ for all $\la\in\La(\Om)$ if and only if $\big [ \widehat{f},\widehat{g} \big ]^{\psi}=0$. In other words, a function $f\in W(\Lt,\ell^1)$ is orthogonal to the space spanned by $\{T_\la g\}_{\la\in\La(\Om)}$ if and only if the bracket product is zero. 

From now on, let
\begin{equation*}
\mathcal{D} = \set{f\in\Lt(\R^m)}{\widehat{f}\in W(L^\infty,\ell^1) \mbox{ and $\widehat{f}$ has compact support in $\R^m$} }
\end{equation*}
It is clear, that $\mathcal{D}$ is a dense subspace of $\Lt(\R^m)$. 
The following observation is easy to verify.
\begin{lemma}\label{lem:lemma1}
Let $\La(\Om)$ be a simple model set, $\Om$ symmetric around the origin, and $\psi\in C_0^\infty(\Om)$. Then, for all $f,g\in W(\Lt,\ell^1)$,
\begin{equation}\label{eq:lemma1}
\M \Big \{\big [ \widehat{f},\widehat{g} \big ]^{\psi} \Big \} = \psi(0) \inner{f}{g}\,.
\end{equation}
\end{lemma}

Observe, that by Lemma~\ref{lem:lemma1}, $\M\Big \{\big [ \widehat{f},\widehat{g} \big ]^{\psi} \Big\} = \M\{[ f,g]^{\psi} \}$. It also follows from Lemma~\ref{lem:lemma1}, that $\M\Big \{\big [ \widehat{f},\widehat{g} \big ]^{\psi} \Big \}$ is independent of $\psi$, as long as $\psi(0)=1$. Moreover, the relation \eqref{eq:lemma1} can be explicitly written as
\begin{equation}\label{eq:relation1}
\lim_{R\rightarrow \infty} \frac{1}{R^m} \int_{B(0,R)} \sum_{\ga^\ast \in \Ga^\ast}\widetilde{w}_\psi(-p^*_2(\ga^\ast))  \big (\widehat{f} \, \overline{\widehat{g}}\big )(t-p^*_1(\ga^\ast)) \,dt = \psi(0) \int_{\R^m} \big (\widehat{f} \, \overline{\widehat{g}} \big )(t)\, dt.
\end{equation}
We make the following useful observation that is in analogy with regular shifts.

\begin{lemma}\label{lem:lemma2}
Let $\La(\Om)$ be a simple model set, $\Om$ symmetric around the origin, and $\psi\in C_0^\infty(\Om)$. For all functions $f,g,h\in W(\Lt,\ell^1)$, we have
\begin{equation*}
\sum_{\la\in\La(\Om)} w_\psi(\la)^2 \, \inner{f}{T_\la g}\inner{T_\la h}{f} = \M\Big\{ \big [ \widehat{f},\widehat{g}\big ]^\psi \cdot \big [ \widehat{h},\widehat{f}\big ]^\psi \Big\}\,.
\end{equation*}
\end{lemma}

\begin{proof}
The bracket product $[\widehat{f},\widehat{g}\big ]^\psi$ is an almost periodic function with Fourier coefficients given by $w_\psi(\la)\inner{f}{T_\la g}$. Using Plancherel Theorem for Fourier series of almost periodic functions and \eqref{eq:Fourier coeff}, we obtain 
\begin{align*}
\sum_{\la\in\La(\Om)} w_\psi(\la)^2\,\inner{f}{T_\la g}\inner{T_\la h}{f} &= \sum_{\la\in\La(\Om)} \M_t\Big \{\big [ \widehat{f},\widehat{g} \big ]^{\psi}(t) e^{2\pi i \la \cdot t} \Big \} \, \overline{\M_t\Big \{\big [ \widehat{f},\widehat{h} \big ]^{\psi}(t) e^{2\pi i \la \cdot t} \Big \} }\\
&= \M\Big\{ \big [ \widehat{f},\widehat{g}\big ]^\psi \cdot \overline{\big [ \widehat{f},\widehat{h}\big ]^\psi }\Big\}\,.
\end{align*}
\end{proof}

Bracket product can be used to find a condition on a family of functions $\{ w_\psi(\la) T_\la g:\la\in\La(\Om)\}$ to be a Bessel sequence.
\begin{lemma}\label{lem:bessel}
Let $\La(\Om)$ be a simple model set, $\Om$ symmetric around the origin, $\psi\in C_0^\infty(\Om)$ and $g \in W(\Lt,\ell^1)$. If for some $B<\infty$,
\begin{equation*}
\text{vol}(\Ga)^{-1} \sum_{\ga^\ast \in \Ga^\ast} \absbig{\widetilde{w}_\psi(-p^*_2(\ga^\ast))} \,\absbig{\widehat{g}(t-p^*_1(\ga^\ast))}^2 \leq B \quad \mbox{for all $t\in\R^m$}\,,
\end{equation*}
then the family $\{ w_\psi(\la) T_\la g:\la\in\La(\Om)\}$ is a Bessel sequence with Bessel bound $B_{g,\psi} = B \cdot \text{vol}(\Ga)^{-1} \norm{\widehat{\psi}}_1$.
\end{lemma}

\begin{proof}
It suffices to verify the claim for $f\in\mathcal{D}$.  By Lemma~\ref{lem:lemma2}, the Bessel condition can be expressed using the mean of the bracket product $\big [\widehat{f},\widehat{g}\big ]^{\psi}$, 
\begin{equation*}
\sum_{\la\in\La(\Om)} w_\psi(\la)^2 \, \abs{\inner{f}{g_\la}} ^2= \M \Big \{ \absbig{\big [\widehat{f},\widehat{g}\big ]^{\psi}}^2 \Big \}\,.
\end{equation*} 
We also observe, that using Cauchy-Schwarz inequality we have
\begin{align*}
\absbig{\big [\widehat{f},\widehat{g} \big ]^{\psi}(t)}^2 &= \absBig{\sum_{\ga^\ast \in \Ga^\ast} \widetilde{w}_\psi(-p^*_2(\ga^\ast))  \widehat{f}(t-p^*_1(\ga^\ast))\overline{\widehat{g}(t-p^*_1(\ga^\ast))}}^2\\
& \leq \Big (\sum_{\ga^\ast \in \Ga^\ast} \absbig{\widetilde{w}_\psi(-p^*_2(\ga^\ast))} \,\absbig{\widehat{f}(t-p^*_1(\ga^\ast))} \,\absbig{\overline{\widehat{g}(t-p^*_1(\ga^\ast))}} \Big )^2\\
&\leq \underbrace{\sum_{\ga^\ast \in \Ga^\ast} \absbig{\widetilde{w}_\psi(-p^*_2(\ga^\ast))} \, \absbig{\widehat{f}(t-p^*_1(\ga^\ast))}^2}_{\big [\widehat{f},\widehat{f}\big ]^{\abs{\psi}}(t)} \, \cdot \, \underbrace{\sum_{\ga^\ast \in \Ga^\ast} \absbig{\widetilde{w}_\psi(-p^*_2(\ga^\ast))}\, \absbig{\widehat{g}(t-p^*_1(\ga^\ast))}^2}_{\big [\widehat{g},\widehat{g}\big ]^{\abs{\psi}}(t)}\,. 
\end{align*}
By the assumption, $\big [\widehat{g},\widehat{g}\big ]^{\abs{\psi}}(t) \leq B$ for a.e. $t\in\R^m$. Therefore,
\begin{align*}
\M \Big \{ \absbig{\big [\widehat{f},\widehat{g}\big ]^{\psi}}^2 \Big \} &= \lim_{R\rightarrow \infty} \frac{1}{R^m} \int_{B(0,R)}   \absbig{\big [\widehat{f},\widehat{g}\big ]^{\psi}(t)}^2 \, dt \leq \lim_{R\rightarrow \infty} \frac{1}{R^m} \int_{B(0,R)}  \big [\widehat{f},\widehat{f}\big ]^{\abs{\psi}}(t) \, \big [\widehat{g},\widehat{g}\big ]^{\abs{\psi}}(t) dt \\
&\leq B  \, \cdot \, \M \Big \{ \big [\widehat{f},\widehat{f}\big ]^{\abs{\psi}} \Big \}\,.
\end{align*}
However, $\M \Big \{ \big [\widehat{f},\widehat{f}\big ]^{\abs{\psi}} \Big \}$ is the $0-$th Fourier coefficient of $ \big [\widehat{f},\widehat{f}\big ]^{\abs{\psi}}$, which can be computed to be $\text{vol}(\Ga)^{-1} \norm{\widehat{\psi}}_1 \norm{f}_2^2$. Putting all calculations together, the claim follows.
\end{proof}

The following result concerning the bracket product will be important in many calculations to follow.
\begin{proposition}\label{prop:ap bracket}
Let $\La(\Om)$ be a simple model set with $\Om$ symmetric around the origin and $\psi\in C_0^\infty(\Om)$ be non-negative. Assume that $g,h\in W(\Lt,\ell^1)$. Then, for $f\in\mathcal{D}$,
\begin{equation*}
F(x,t) = \big [\widehat{T_x f},\widehat{g}\big ]^\psi(t) \cdot \big [ \widehat{h},\widehat{T_x f}\big ]^\psi(t)\,, \quad (x,t)\in \R^m \times \R^m
\end{equation*}
is an almost periodic function.
\end{proposition}

\begin{proof}
Let $f\in\mathcal{D}$. For $\eta\in\Ga^*$, we define $F_g^\eta$ to be such that $\widehat{F_g^\eta} = T_{p_1^*(\eta)}\widehat{f} \cdot \widehat{g}$, and $F_h^\eta$ such that  $\widehat{F_h^\eta} = \widehat{f} \cdot T_{p_1^*(\eta)} \widehat{h}$. Moreover, for $\eta\in \Ga^\ast$, we define $\Psi_\eta$ such that $\widehat{\Psi_\eta} = \widehat{\psi} \cdot T_{p^*_2(\eta)} \widehat{\psi}$. Then each $\Psi_\eta$ belongs to $C_0^\infty(\R)$ and is compactly supported on $\Om+\Om$, and we can define $\widetilde{w}_{\Psi_\eta}$ as $\widetilde{w}_{\Psi_\eta} (p^*_2(\ga^*))=\text{vol}(\Ga)^{-1} \widehat{\Psi_\eta}(p^*_2(\ga^*))$, for all $\ga^*\in\Ga^*$, as in \eqref{eq:Fourier weight}. Then, by the change of index, we have
\begin{align}
&F(x,t) =\sum_{\mu,\theta\in \Ga^\ast} \widetilde{w}_\psi(-p^*_2(\mu)) \widetilde{w}_\psi(-p^*_2(\theta)) \big( \widehat{f} \,\overline{\widehat{g}}\big )(t-p^*_1(\mu)) \big(\overline{\widehat{f}}\widehat{h} \big )(t-p^*_1(\theta)) e^{2\pi i x \cdot (p^*_1(\mu) - p^*_1(\theta)) } \nonumber \\
&= \sum_{\eta,\mu\in \Ga^\ast} \widetilde{w}_\psi(-p^*_2(\mu)) \widetilde{w}_\psi(-p^*_2(\mu)-p^*_2(\eta)) \big( \widehat{f} \,\overline{\widehat{g}}\big )(t-p^*_1(\mu)) \big( \overline{\widehat{f}}\widehat{h} \big )(t-p^*_1(\mu)-p^*_1(\eta))\, e^{-2\pi i x \cdot p^*_1(\eta)} \nonumber \\
&=   \sum_{\eta\in \Ga^\ast}  \sum_{\mu \in \Ga^\ast}\text{vol}^{-2}(\Ga) \widehat{\Psi_\eta}(-p^*_2(\mu)) \widehat{F_h^\eta}(t-p^*_1(\mu)) \overline{\widehat{F_g^\eta}(t-p^*_1(\mu))} \, e^{-2\pi i x \cdot p^*_1(\eta)}\nonumber \\
&= \text{vol}(\Ga)^{-1} \sum_{\eta\in \Ga^\ast}  \sum_{\mu \in \Ga^\ast} \widetilde{w}_{\Psi_\eta}(-p^*_2(\mu)) \widehat{F_h^\eta}(t-p^*_1(\mu)) \overline{\widehat{F_g^\eta}(t-p^*_1(\mu))} \, e^{-2\pi i x \cdot p^*_1(\eta)}\nonumber\\
&= \text{vol}(\Ga)^{-1}  \sum_{\eta\in \Ga^\ast}  \big [ \widehat{F_h^\eta}, \widehat{F_g^\eta} \big ]^{\Psi_\eta} (t) \, e^{-2\pi i x \cdot p^*_1(\eta)}\,. \label{eq:F}
\end{align}
Since $f\in\mathcal{D}$ and $g,h\in W(\Lt,\ell^1)$, we have that $F_g^\eta,\widehat{F_g^\eta}\in W(L^\infty,\ell^1)$ and $F_h^\eta,\widehat{F_h^\eta}\in W(L^\infty,\ell^1)$, and the bracket products $\big [ \widehat{F_h^\eta}, \widehat{F_g^\eta} \big ]^{\Psi_\eta}$ are well defined almost periodic functions. They are represented by their Fourier series, 
\begin{align*}
\big [ \widehat{F_h^\eta}, \widehat{F_g^\eta} \big ]^{\Psi_\eta}(t) &= \sum_{\la\in\La(\Om+\Om)} w_{\Psi_\eta}(\la) \inner{F_h^\eta}{T_\la F_g^\eta} e^{-2\pi i\, t\cdot \la}\\
&= \sum_{\ga\in \Ga} \Psi_\eta(p_2(\ga)) \inner{F_h^\eta}{T_{p_1(\ga)} F_g^\eta} e^{-2\pi i\, t\cdot p_1(\ga)}\,,
\end{align*}
and we can write
\begin{equation*}\label{eq:fourier series F}
F(x,t) =  \text{vol}(\Ga)^{-1}  \sum_{\eta\in \Ga^\ast} \sum_{\ga\in \Ga} \Psi_\eta(p_2(\ga)) \inner{F_h^\eta}{T_{p_1(\ga)} F_g^\eta} e^{-2\pi i\, t\cdot p_1(\ga)} \,e^{-2\pi i x \cdot p^*_1(\eta)}\,.
\end{equation*}
Since,
\begin{align*}
\Psi_\eta(p_2(\ga)) &= \big(\conv{\psi}{M_{p_2^*(\eta)}}\big )(p_2(\ga)) = \innerbig{\psi}{T_{p_2(\ga)}M_{p_2^*(\eta)}\psi^*}\\ \innerbig{F_h^\eta}{T_{p_1(\ga)} F_g^\eta} &= \innerbig{\widehat{F_h^\eta}}{M_{-p_1(\ga)}\widehat{F_g^\eta}} = \innerbig{\widehat{f}\overline{\widehat{g}}}{M_{-p_1(\ga)}T_{p_1^*(\eta)} \widehat{f}\overline{\widehat{h}}}\\
&= \innerbig{\conv{f}{g^*}}{T_{p_1(\ga)}M_{p_1^*(\eta)}\big( \conv{f}{h^*}\big)}\,,
\end{align*}
where $\psi^*$ is an involution $\psi^(x)=\overline{\psi(-x)}$, we have
\begin{align*}
F(x,t) &=\text{vol}(\Ga)^{-1}  \sum_{\eta\in \Ga^\ast} \sum_{\ga\in \Ga} \inner{\psi}{T_{p_2(\ga)}M_{p_2^*(\eta)}\psi^*} \inner{\conv{f}{g^*}}{T_{p_1(\ga)}M_{p_1^*(\eta)} \big (\conv{f}{h^*}\big)} e^{-2\pi i\, t\cdot p_1(\ga)} \,e^{-2\pi i x \cdot p^*_1(\eta)}\\
&= \text{vol}(\Ga)^{-1}  \sum_{\eta\in \Ga^\ast} \sum_{\ga\in \Ga}   \innerBig{\big (\conv{f}{g^*} \big )\otimes \psi}{T_\ga M_{\eta} \big( \big (\conv{f}{h^*} \big )\otimes \psi^* \big)}  e^{-2\pi i\, t\cdot p_1(\ga)} \,e^{-2\pi i x \cdot p^*_1(\eta)}\,.
\end{align*}
For $f\in\mathcal{D}$ and $g,h\in W(\Lt,\ell^1)$, we have that $(\conv{f}{g^*}),(\conv{f}{h^*})\in W(L^\infty,\ell^1) \subset L^1(\R^m)$ and, since $\widehat{f}$ is compactly supported, both $\F(\conv{f}{g^*})=\widehat{f}\overline{\widehat{g}}$ and  $\F(\conv{f}{h^*})=\widehat{f}\overline{\widehat{h}}$ are also compactly supported. That means (see Thm.~$3.2.17$ in \cite{fezi98}), that $(\conv{f}{g^*}),(\conv{f}{h^*})\in M^1(\R^m)$. Also, $\psi\in M^1(\R)$, because $C_0^\infty(\Om) \subset \mathcal{S}(\R)\subset M^1(\R)$. Hence, both tensor products, $\big (\conv{f}{g^*} \big )\otimes \psi$ and $\big (\conv{f}{h^*} \big )\otimes \psi^*$ lie in $M^1(\R^m\times \R)$. By Proposition~\ref{prop:STFT},
\begin{equation*}
\V_{(\conv{f}{h^*})\otimes \psi^*}\left (\big (\conv{f}{g^*} \big )\otimes \psi \right) \in W(C_0,\lo)((\R^m\times \R)\times (\R^m\times \R))
\end{equation*}
By the norm estimate \eqref{eq:norm_estimate}, for $\La  =\Ga\times \Ga^*$, the coefficients in the series defining $F(x,t)$ are in $\ell^1(\Ga\times\Ga^*)$. That means that $F$ equals a generalized trigonometric polynomial, and therefore is almost periodic.
\end{proof}

We need one more result that will be an important tool in the characterization of tight frames and dual frames.

\begin{proposition}\label{prop:function N}
Let $\La(\Om)$ be a simple model set, $\Om$ symmetric around the origin and $\psi\in C_0^\infty(\Om)$ a non-negative function. Assume that $g_k,h_k\in W(\Lt,\ell^1)$ for every $k\in\mathcal{K}$, and $\sum_{k\in\mathcal{K}} \abs{\widehat{g_k}(t)}^2 \leq B_g$, $\sum_{k\in\mathcal{K}} \abs{\widehat{h_k}(t)}^2 \leq B_h$, with $B_g,B_h>0$.
Then, for every $f\in\mathcal{D}$, the function
\begin{equation}\label{eq:function G}
\mathcal{N}^\psi(x) = \sum_{k\in\mathcal{K}} \sum_{\la\in\La(\Om)} w_\psi(\la)^2 \, \inner{T_x f}{T_\la g_k} \inner{T_\la h_k}{T_x f}
\end{equation}
is almost periodic and coincides pointwise with its Fourier series $\sum_{\eta \in \Ga^\ast} \widehat{\mathcal{N}^\psi}(\eta) e^{-2\pi i p^*_1(\eta) \cdot x}$, with
\begin{equation*}
\widehat{\mathcal{N}^\psi}(\eta) = \text{vol}(\Ga)^{-1}\,\widehat{\psi^2} (-p_2^*(\eta)) \int_{\R^m} \widehat{f}(t) \overline{\widehat{f}(t-p^*_1(\eta))} \left (\sum_{k\in\mathcal{K}} \overline{\widehat{g_k}(t)}\, \widehat{h_k}(t-p^*_1(\eta))\right ) \, dt \,,
\end{equation*}
where $\eta\in \Ga^\ast$ and the integral converges absolutely.
\end{proposition}

\begin{proof}
Let $f\in\mathcal{D}$.  For $k\in\mathcal{K}$ fixed, let
\begin{equation*}
\mathcal{N}^\psi_k (x) = \sum_{\la\in\La(\Om)} w_\psi(\la)^2\, \inner{T_x f}{T_\la g_k} \inner{T_\la h_k}{T_x f}\,.
\end{equation*}
By Lemma~\ref{lem:lemma2}, with $g=g_k$ and $h=h_k$, we have
\begin{equation*}
\mathcal{N}^\psi_k(x) = \M_t\Big\{ \big [ \widehat{T_x f},\widehat{g_k}\big ]^\psi(t) \cdot \big [ \widehat{h_k},\widehat{T_x f}\big ]^\psi(t) \Big\} \,,
\end{equation*}
and by Proposition~\ref{prop:ap bracket} and Theorem~\ref{thm:ap in 2 variables}, $\mathcal{N}^\psi_k(x)$ is almost periodic. For $\eta\in\Ga^*$, we define $F_{g_k}^\eta$ to be such that $\widehat{F_{g_k}^\eta} = T_{p_1^*(\eta)}\widehat{f}\cdot \widehat{g_k}$, and $F_{h_k}^\eta$ such that  $\widehat{F_{h_k}^\eta} = \widehat{f}\cdot T_{p_1^*(\eta)} \widehat{h_k}$. Moreover, for $\eta\in \Ga^\ast$, we define $\Psi_\eta$ such that $\widehat{\Psi_\eta} = \widehat{\psi} \cdot T_{p^*_2(\eta)} \widehat{\psi}$. Then each $\Psi_\eta$ belongs to $C_0^\infty(\R)$ and is compactly supported on $\Om+\Om$, and we can define $\widetilde{w}_{\Psi_\eta}$ as $\widetilde{w}_{\Psi_\eta} (p^*_2(\ga^*))=\text{vol}(\Ga)^{-1} \widehat{\Psi_\eta}(p^*_2(\ga^*))$, for all $\ga^*\in\Ga^*$, as in \eqref{eq:Fourier weight}.

Then, by the same calculations as in \eqref{eq:F}, the function $\mathcal{N}^\psi_k(x)$ becomes
\begin{align}
\mathcal{N}^\psi_k(x) &= \text{vol}(\Ga)^{-1} \sum_{\eta\in \Ga^*} \M_t \Big \{ \big [\widehat{F_{h_k}^\eta}, \widehat{F_{g_k}^\eta}\big ]^{\Psi_\eta}(t) \Big\} e^{-2\pi i x \cdot p_1^*( \eta) } \nonumber \\
&=  \text{vol}(\Ga)^{-1} \sum_{\eta\in \Ga^*} \Psi_\eta(0) \left (\int_{\R^m} \widehat{F_{h_k}^\eta}(t) \overline{\widehat{F_{g_k}^\eta}(t)} \,dt \right ) e^{-2\pi i x \cdot p_1^*(\eta)} \quad \mbox{by Lemma~\ref{lem:lemma1}} \nonumber \\
&= \text{vol}(\Ga)^{-1} \sum_{\eta\in \Ga^*} \widehat{\psi^2} (-p_2^*(\eta)) \left ( \int_{\R^m}  \widehat{f}(t) \overline{\widehat{f}(t-p_1^*(\eta))} \overline{\widehat{g_k}(t)}\widehat{h_k}(t-p_1^*(\eta)) \, dt \right ) e^{-2\pi i x \cdot p_1^*(\eta) }\, \label{eq:fourier series N_k}
\end{align}
By Cauchy-Schwarz inequality, we have
\begin{equation*}
\absBig{\sum_{k\in\mathcal{K}} \overline{\widehat{g_k}(t)}\widehat{h_k}(t-p_1^*(\eta))} \leq  \left ( \sum_{k\in\mathcal{K}} \absbig{\widehat{g_k}(t)}^2\right )^{1/2} \left ( \sum_{k\in\mathcal{K}} \absbig{\widehat{h_k}(t-p_1^*(\eta))}^2\right )^{1/2} \leq B_g^{1/2} B_h^{1/2} \,,
\end{equation*}
and the function $\sum_{k\in\mathcal{K}} \overline{\widehat{g_k}(t)}\widehat{h_k}(t-p_1^*(\eta))$ is locally integrable.
Therefore, since $f\in\mathcal{D}$, using Lebesgue Dominated Convergence Theorem, we have
\begin{align}
&\mathcal{N}^\psi(x) = \lim_{M\rightarrow \infty} \sum_{k\in \mathcal{K}_M} \mathcal{N}^\psi_k(x) \nonumber \\
&= \sum_{\eta\in \Ga^*} \left ( \text{vol}(\Ga)^{-1} \widehat{\psi^2} (-p_2^*(\eta)) \int_{\R^m} \widehat{f}(t) \overline{\widehat{f}(t-p_1^*(\eta))} \left (\sum_{k\in\mathcal{K}} \overline{\widehat{g_k}(t)}\, \widehat{h_k}(t-p_1^*(\eta))\right )\, dt \right ) e^{-2\pi i x \cdot p_1^*(\eta) }\,, \label{eq:series}
\end{align}
where $\mathcal{K}_M$ are finite subsets of $\mathcal{K}$ such that $\mathcal{K}_M \subseteq \mathcal{K}_{M+1}$, and the above holds with absolute convergence of the integral. By the uniqueness of the Fourier series, \eqref{eq:series} is the Fourier series of $\mathcal{N}^\psi(x)$.
\end{proof}


\section{Main Result}\label{sec:main}

At the beginning of Section~\ref{sec:not}, in Proposition~\ref{prop:operator shifts}, we showed the relationship between frame operators $S_g^{\La}$ and $S_g^{\La-x}$ where $\La$ was any relatively separated subset of $\R^m$. The same holds in particular for any $\La\in X(\La(\Om))$ and $x\in\R^m$. We show below that for $\La_1,\La_2\in X(\La(\Om))$, if $\La_1$ and $\La_2$ are $(R,\epsilon)$-close then, for every $f\in \Lt(\R^m)$, $S_g^{\La_1}f$ and $S_g^{\La_2}f$ are close. In other words, $S_g^{\La}$ is continuous, with respect to $\La$, in the strong operator topology. We formulate the result for the mixed frame operators $S_{g,h}^{\La}$. The following two Propositions are analogous to the results of Kreisel for Gabor frames \cite{K16}.

\begin{proposition}\label{prop:continuity}
Let $\La(\Om)$ be a simple model set and $g_k, h_k \in W(\Lt,\ell^1)$ for every $k\in\mathcal{K}$. Assume that $\{T_\la g_k:k\in\mathcal{K},\la\in\La'\}$ and  $\{T_\la h_k:k\in\mathcal{K},\la\in\La'\}$ are  Bessel sequences for each $\La'\in X(\La(\Om))$. If $\La_n$ converges to $\La$ in $X(\La(\Om))$, then the mixed frame operators $S_{g,h}^{\La_n}$ converges to  $S_{g,h}^{\La}$ in the strong operator topology. 
\end{proposition} 

\begin{proof}
Let $f\in \mathcal{D}$ and $\epsilon >0$. Since $\text{rel}(\La') = \text{rel}(\La(\Om))$ for all $\la'\in X(\La(\Om))$ and $\{T_\la g_k:k\in\mathcal{K},\la\in\La'\}$ and  $\{T_\la h_k:k\in\mathcal{K},\la\in\La'\}$ are  Bessel sequences for all $\La'\in X(\La(\Om))$, we can choose $R>0$ large enough so that the sum $S_{g,h}^{\La'}f$ is arbitrarily small outside of the cube $B(0,R)$ independent of $\La' \in X(\La(\Om))$. That is
\begin{equation}\label{eq:bound on the mixed FO}
\left\| \sum_{k\in\mathcal{K}} \sum_{\la\in \La' \setminus B(0,R)} \inner{f}{T_\la g_k}T_\la h_k \right \|_2 < \frac{\epsilon}{4}
\end{equation}
On the other hand, since $\La_n$ converges to  $\La$, we can choose $N$ so that for all $n\geq N$, $\La_n$ agrees with $\La$ on $B(0,R)$ up to a small translation, so that
\begin{equation*}
\left \| \sum_{k\in\mathcal{K}} \sum_{\la\in \La \cap B(0,R)} \inner{f}{T_\la g_k}T_\la h_k - \sum_{k\in\mathcal{K}} \sum_{\la\in \La_n \cap B(0,R)} \inner{f}{T_\la g_k}T_\la h_k \right \|_2 < \frac{\epsilon}{2}\,.
\end{equation*}
Then for all $n\geq N$ we have
\begin{align*}
\normBig{S_{g,h}^{\La} f- S_{g,h}^{\La_n}f}_2 &\leq \left \|  \sum_{k\in\mathcal{K}} \sum_{\la\in \La \setminus B(0,R)} \inner{f}{T_\la g_k}T_\la h_k - \sum_{k\in\mathcal{K}} \sum_{\la\in \La_n \setminus B(0,R)} \inner{f}{T_\la g_k}T_\la h_k \right \|_2 + \frac{\epsilon}{2}\\
&\leq \left \| \sum_{k\in\mathcal{K}} \sum_{\la\in \La \setminus B(0,R)} \inner{f}{T_\la g_k}T_\la h_k \right \|_2 + \left \|\sum_{k\in\mathcal{K}} \sum_{\la\in \La_n \setminus B(0,R)} \inner{f}{T_\la g_k}T_\la h_k \right \|_2 + \frac{\epsilon}{2}\\
&< \frac{\epsilon}{4} + \frac{\epsilon}{4} + \frac{\epsilon}{2} = \epsilon
\end{align*}
where in the last line we used \eqref{eq:bound on the mixed FO}.
\end{proof}

Using continuity of the frame operator, we can improve Proposition~\ref{prop:operator shifts} to hold for all the point sets in the hull $X(\La(\Om))$. 

\begin{proposition}\label{prop:all shifts}
Let $\La(\Om)$ be a simple model set and $g_k\in W(\Lt,\ell^1)$ for each $k\in\mathcal{K}$. Suppose that $\{T_\la g_k:k\in\mathcal{K},\la\in\La(\Om)\}$ is a frame (Bessel sequence) for $\Lt(\R^m)$ with upper and lower frame bounds $A_g$ and $B_g$, respectively.  Then, for any $\La \in X(\La(\Om))$, $\{T_\la g_k:k\in\mathcal{K},\la\in\La\}$ is a frame (Bessel sequence) for $\Lt(\R^m)$ with the upper and lower frame bounds $A_g$ and $B_g$, respectively.
\end{proposition}

\begin{proof}
Let $\La\in X(\La(\Om))$. Since the set $\{\La(\Om)-x:x\in\R^m\}$ is dense in $X(\La(\Om))$, we can find a sequence of shifts $\La_n = \La(\Om)-x_n$ converging to $\La$ in the local topology. Then, by Proposition~\ref{prop:continuity} we have $S_g^{\La_n}$ converging to $S_g^{\La}$ in the strong operator topology. Due to Proposition~\ref{prop:operator shifts}, the frame bounds for the frames $\{T_\la g_k:k\in\mathcal{K},\la\in\La(\Om) - x_n\}$ are all equal, so $S_g^{\La}$ also satisfies the same frame bounds. In particular, $S_g^{\La}$ is bounded from below, and $\{T_\la g_k:k\in\mathcal{K},\la\in\La\}$ is a frame.
\end{proof}

The last Proposition tells us that if we know that $\{T_\la g_k:k\in\mathcal{K},\la\in\La(\Om)\}$ is a frame, or Bessel sequence, for $\Lt(\R^m)$, then we know it for the whole collection of point sets from $X(\La(\Om))$. Applying results from Section~\ref{sec:bracket} we will specify conditions on windows $g_k$ so that the family $\{T_\la g_k:k\in\mathcal{K},\la\in\La(\Om)\}$ is a tight frame, respectively, a dual frame. 

In order to do so, fix $f\in\mathcal{D}$ and let $g_k,h_k\in W(\Lt,\ell^1)$ be a collection of functions, $k\in\mathcal{K}$. We define a function $\widetilde{\mathcal{N}} : X(\La(\Om)) \rightarrow \C$ through the mixed frame operator, as
\begin{equation*}
\widetilde{\mathcal{N}}(\La) = \inner{S_{g,h}^{\La}f}{f}\,.
\end{equation*}
As long as $\{T_\la g_k:k\in\mathcal{K},\la\in\La \}$ and $\{T_\la h_k:k\in\mathcal{K},\la\in\La\}$ are Bessel sequences, $\widetilde{\mathcal{N}}$ is well defined.  Since $S_{g,h}^{\La}$ is continuous, in the strong operator topology, over $X(\La(\Om))$, the function $\widetilde{\mathcal{N}}$ is continuous. We can define from $\widetilde{\mathcal{N}}$ a function $\mathcal{N} :\R^m\rightarrow \C$, by
\begin{equation}\label{eq:main function}
\mathcal{N}(x) = \widetilde{\mathcal{N}}(\La(\Om)-x)\,,
\end{equation}
and since $\widetilde{\mathcal{N}}$ is continuous, $\mathcal{N}$ is local with respect to $\La(\Om)$. As was shown in Section~\ref{sec:local functions}, it has a Fourier expansion
\begin{equation*}
\mathcal{N}(x) = \sum_{\eta\in\Ga^*} \widehat{\mathfrak{N}}(\eta) e^{-2\pi i p_1^*(\eta) \cdot x}
\end{equation*}
where
\begin{equation}\label{eq:fourier coeff}
\widehat{\mathfrak{N}}(\eta)  = \lim_{R\rightarrow \infty} \frac{1}{R^m} \int_{B(0,R)} \mathcal{N}(x) e^{2\pi i x \cdot p_1^*(\eta)}\, dx\,.
\end{equation}
Applying the tools developed in Section~\ref{sec:bracket}, we will be able to compute the Fourier coefficients $\widehat{\mathfrak{N}}(\eta)$ of $\mathcal{N}$. We start with a simple result, that we deduce from Proposition~\ref{prop:ap bracket} and its proof, in analogy with \cite{J98},\cite{L02} for the case of regular shifts.

\begin{proposition}\label{prop:bessel}
Let $\La(\Om)$ be a simple model set with $\Om$ symmetric around the origin, $\psi\in C_0^\infty(\Om)$ a non-negative function and $g_k\in W(\Lt,\ell^1)$ for every $k\in\mathcal{K}$. If $\{T_\la g_k:k\in\mathcal{K},\la\in\La(\Om)\}$ is a Bessel sequence with upper frame bound $B_g$, then $\{w_\psi(\la)T_\la g_k:k\in\mathcal{K},\la\in\La(\Om)\}$ is also a Bessel sequence with upper frame bound $B_{g,\psi} = \norm{\psi}_\infty^2 B_g$. Moreover,
\begin{equation}\label{eq:bessel}
\text{vol}(\Ga)^{-1} \, \widehat{\psi^2}(0)\, \sum_{k\in\mathcal{K}} \absbig{\widehat{g_k}(t)}^2 \leq B_{g,\psi} \quad \mbox{for a.e. $t\in\R^m$}\,.
\end{equation}
\end{proposition}

\begin{remark}
In particular, choosing $\psi\in C_0^\infty(\Om)$ such that $\norm{\psi}_2^2=\abs{\Om}$ and $\norm{\psi}_\infty\leq 1$, it follows from Proposition~\ref{prop:bessel}, that the Bessel property of $\{T_\la g_k:k\in\mathcal{K},\la\in\La(\Om)\}$ implies $\text{vol}(\Ga)^{-1} \abs{\Om} \sum_{k\in\mathcal{K}} \absbig{\widehat{g_k}(t)}^2 \leq B_g$. However, for simple model sets, the density $D(\La(\Om))$ equals $\text{vol}(\Ga)^{-1} \abs{\Om}$, and the inequality becomes
\begin{equation*}
D(\La(\Om))  \sum_{k\in\mathcal{K}} \absbig{\widehat{g_k}(t)}^2 \leq B_g\,.
\end{equation*}
\end{remark}

\begin{proof}
Let $f\in\Lt(\R^m)$. Then 
\begin{equation*}
\sum_{k\in\mathcal{K}} \sum_{\la\in\La(\Om)} w_\psi(\la)^2 \, \absbig{\inner{f}{T_\la g_k}}^2 \leq \sup_{\la\in\La(\Om)}  w_\psi(\la)^2 \,\sum_{k\in\mathcal{K}} \sum_{\la\in\La(\Om)} \absbig{\inner{f}{T_\la g_k}}^2 \leq \norm{\psi}_\infty^2 B_g \norm{f}_2^2\,,
\end{equation*}
and the last inequality follows from $\{T_\la g_k:k\in\mathcal{K},\la\in\La(\Om)$ being a Bessel sequence for $\Lt(\R^m)$.

For the last part of the Proposition, it is sufficient to prove the statement for $f\in\mathcal{D}$. Since $\{w_\psi(\la) T_\la g_k:k\in\mathcal{K}, \la\in\La(\Om)\}$ is a Bessel sequence, then for every finite subset $\mathcal{K}_M$ of $\mathcal{K}$, and every $x\in\R^m$, we have, 
\begin{equation}\label{eq:int}
\sum_{k\in\mathcal{K}_M} \sum_{\la\in\La(\Om)} w_\psi(\la)^2\, \absbig{\inner{T_x f}{T_\la g_k}}^2  \leq   B_{g,\psi} \norm{T_x f}_2^2 = B_{g,\psi} \norm{f}_2^2\,,
\end{equation}
where $\mathcal{K}_M \subseteq \mathcal{K}_{M+1}$. By Lemma~\ref{lem:lemma2}, for $g=g_k$ and $h=g_k$, this becomes
\begin{equation}\label{eq:eq bessel}
\sum_{k\in\mathcal{K}_M} \M_t\Big \{  \absBig{\big [ \widehat{T_x f},\widehat{g_k}\big ]^\psi(t)}^2 \Big \} = \sum_{k\in\mathcal{K}_M} \M_t\Big \{  \big [ \widehat{T_x f},\widehat{g_k}\big ]^\psi(t) \cdot \big [ \widehat{g_k},\widehat{T_x f}\big ]^\psi(t)\Big \} \leq   B_{g,\psi} \norm{f}_2^2\,.
\end{equation}
Since, by Proposition~\ref{prop:ap bracket} for each $k\in\mathcal{K}$, $F(x,t) = \big [ \widehat{T_x f},\widehat{g_k}\big ]^\psi(t) \cdot \big [ \widehat{g_k},\widehat{T_x f}\big ]^\psi(t)$ is almost periodic in $(x,t)\in\R^m\times \R^m$, we have that, by Theorem~\ref{thm:ap in 2 variables}, $\M_t\Big \{ \absBig{\big [ \widehat{T_x f},\widehat{g_k}\big ]^\psi(t)}^2 \Big \}$ is almost periodic in $x\in\R^m$. A finite sum of almost periodic functions is almost periodic, and we can take the mean value over $x\in\R^m$ on both sides of \eqref{eq:eq bessel}:
\begin{equation}\label{eq:bessel bound}
\sum_{k\in\mathcal{K}_M} \M_x \M_t \Big \{  \big [ \widehat{T_x f},\widehat{g_k}\big ]^\psi(t) \cdot \big [ \widehat{g_k},\widehat{T_x f}\big ]^\psi(t)\Big \} \leq B_{g,\psi} \norm{f}_2^2\,.
\end{equation}
The interchange of summation and integration is justified since the summation is over a finite set. From the theory of almost periodic functions, we know that  $\M_x \M_t \{F(x,t)\} =\M_t \M_x\{F(x,t)\} = \M_{x,t} \{F(x,t)\} $ for all almost periodic functions in $(x,t)$, and, by definition, the mean value over $(x,t)$ is a $0-$th Fourier coefficient of $F(x,t)$. 
For $\eta\in\Ga^*$, we define $F_g^\eta$ to be such that $\widehat{F_{g}^\eta} = T_{p_1^*(\eta)}\widehat{f}\cdot \widehat{g}$, and , by \eqref{eq:fourier series F}
\begin{align}\label{eq:BB}
\sum_{k\in\mathcal{K}_M} \M_x \M_t \Big \{  \big [ \widehat{T_x f},\widehat{g_k}\big ]^\psi(t) \cdot \big [ \widehat{g_k},\widehat{T_x f}\big ]^\psi(t)\Big \}&= \sum_{k\in\mathcal{K}_M} \text{vol}(\Ga)^{-1}\Psi_0(0) \inner{F_{g}^0}{F_{g}^0}\\
&= \text{vol}(\Ga)^{-1} \sum_{k\in\mathcal{K}_M} \widehat{\psi^2}(0) \int_{\R^m} \absbig{\widehat{f}(t)}^2 \absbig{\widehat{g_k}(t)}^2\, dt \nonumber \\
&=  \text{vol}(\Ga)^{-1} \widehat{\psi^2}(0) \int_{\R^m}\absbig{ \widehat{f}(t)}^2 \sum_{k\in\mathcal{K}_M} \absbig{\widehat{g_k}(t)}^2 \, dt\,. \nonumber
\end{align}
Now, let $G_{0,M}(t) = \text{vol}(\Ga)^{-1} \widehat{\psi^2}(0) \sum_{k\in\mathcal{K}_M} \absbig{\widehat{g_k}(t)}^2$. Then $G_{0,M}\in L^1(\R^m)$, since it is a finite sum of $\Lt(\R^m)$ functions, and we denote by $L_M$ a set of Lebesgue points of $G_{0,M}$.  

Let $s$ be a bounded compactly supported function given by $s(t) = \mathds{1}_{B(0,1)}(t)$, where $B(0,1) = [0,1]^m$ is a ball in $\R^m$ centered at zero. We form a sequence of functions $s_n(t) = \sqrt{n} \,s(n\cdot t) = \sqrt{n}\, \mathds{1}_{B(0,1/n)}(t)$. Then $\{ s_n^2\}_{n\in\N}$ is an approximate identity, and for each Lebesgue point $t_0\in L_M$ of $G_{0,M}$, we have
\begin{align}
\lim_{n\rightarrow \infty}\big (s_n^2 \ast G_{0,M}\big )(t_0) &= \lim_{n\rightarrow \infty}\int_{\R^m} s_n^2(t_0-t) G_{0,M}(t)\, dt \nonumber \\
&=  \lim_{n\rightarrow \infty} n^m  \int_{\abs{t_0-t}\leq (1/n)^m} G_{0,M}(t)\,dt \nonumber \\
&= G_{0,M}(t_0)\,. \label{eq:aprox identity0}
\end{align}
Therefore, by letting $\widehat{f_n}(t)=s_n(t_0-t)$, we have $f_n\in\mathcal{D}$ with $\norm{f_n}_2=1$, and by \eqref{eq:aprox identity0},
\begin{equation*}
\lim_{n\rightarrow \infty}  \int_{\R^m} \absbig{\widehat{f_n}(t)}^2 G_{0,M}(t)\,dt = G_{0,M}(t_0) \quad \mbox{for all $t_0\in L_M$.}
\end{equation*}
Therefore, by \eqref{eq:BB}, $G_{0,M}(t_0) \leq B_{g,\psi}$ for all $t_0\in L_M$. Since,
\begin{equation*}
\text{vol}(\Ga)^{-1} \, \widehat{\psi^2}(0)\, \sum_{k\in\mathcal{K}} \absbig{\widehat{g_k}(t)}^2 = \lim_{M\rightarrow \infty} G_{0,M}(t)
\end{equation*}
we obtained the desired result for all $t$ in the intersection of all $L_M$, which is a dense set in $\R^m$.
\end{proof}

Before we proceed further we introduce a sequence of auxiliary functions that will be crucial in proving our results. The following Lemma is a particular case of  Prop~$3.6$ in \cite{BM00}. 
\begin{lemma}\label{lemma:psi}
Let $0<\epsilon<1$, $\Om$ be a compact subset of $\R$ symmetric around the origin, $\widetilde{\Om} = (1-\epsilon)\Om$ and $f_s =  \frac{\mathds{1}_{\epsilon^{s}\widetilde{\Om}}}{\abs{\epsilon^{s}\widetilde{\Om}}}$ for $s\in \N$. Then,
\begin{itemize}
\item[(i)] the infinite convolution product $\Convtext_{s=0}^{\infty} \, f_s$ converges in $L^1(\R)$ and defines a non-negative smooth function $\psi = \Convtext_{s=0}^{\infty} \, f_s$ compactly supported on $\Om$, with $\norm{\psi}_1 = 1$;
\item[(ii)] the Fourier transform of $\psi$ is the smooth function $\widehat{\psi} = \prod_{s=0}^{\infty} \,\widehat{f_s} $, with uniform convergence of the product, where $\widehat{f_s}(t) = \text{sinc}(t\,\abs{\epsilon^{s} \widetilde{\Om}})$ and $\text{sinc}(t) = \frac{\sin(\pi t)}{\pi t}$.
\end{itemize}
\end{lemma}

Now, for $0<\epsilon<1$ and $\Om$ a compact subset of $\R$ symmetric around the origin, we define a sequence of compact sets $\Om_n = (1-\epsilon^n)\Om$. The sets $\Om_n$ are increasing and $\overline{\bigcup_n \Om_n} = \Om$. Let $\psi_n$ be an infinite convolution product
\begin{equation}\label{eq:function psi_n}
\psi_n= \Conv_{s=0}^{\infty} \, \frac{\mathds{1}_{\epsilon^{ns}\Om_n}}{\abs{\epsilon^{ns}\Om_n}} = 
\frac{\mathds{1}_{\Om_n}}{\abs{\Om_n}} \ast \left ( \Conv_{s=1}^{\infty} \, \frac{\mathds{1}_{\epsilon^{ns}\Om_n}}{\abs{\epsilon^{ns}\Om_n}}\right)\,.
\end{equation}
Then, by Lemma~\ref{lemma:psi}, each $\psi_n$ is well defined and forms a  sequence of $C_0^\infty(\Om)$ non-negative functions, with Fourier transform of $\psi_n$ being
\begin{equation}\label{eq:function Fourier psi_n}
\widehat{\psi_n}(t) = \prod_{s=0}^{\infty} \,\text{sinc}(t\,\abs{\epsilon^{ns} \Om_n}) = 
\text{sinc}(t\,\abs{\Om_n})  \cdot \prod_{s=1}^{\infty} \, \text{sinc}(t\,\abs{\epsilon^{ns} \Om_n}) \,,.
\end{equation}

\begin{lemma}\label{lemma:psi_n}
With the above notation, 
\begin{itemize}
\item[(i)] the sequence $\{\psi_n\}_{n=1}^{\infty}$ converges pointwise to $\frac{\mathds{1}_\Om}{\abs{\Om}}$ on $\Om\setminus \partial \Om$, where $\partial \Om$ is the boundary of $\Om$;
\item[(ii)] the sequence of Fourier transforms, $\{\widehat{\psi_n}\}_{n=1}^{\infty}$, converges uniformly.
\end{itemize}
\end{lemma}

\begin{proof}
Let $t_0\in \Om\setminus \partial \Om$ and $\delta>0$. By the properties of the sets $\Om_n$ we have: $\Om_m \subset \Om_n \subset \Om$ and $\epsilon^{ns} \Om_n \subset \epsilon^{ms} \Om_m$ for $n\geq m$ and $\epsilon^{ns} \Om_n \subset \epsilon^{n(s+1)} \Om_n$ for all $n$ and $s$. Then, there exists $N\geq 0$, such that for all $n\geq N$, $\abs{\Om \setminus \Om_n} < \delta \abs{\Om}^2$ and $t_0\in\Om_n$. That means, for $n\geq N$
\begin{equation*}
\left [ \left (\sum_{s=1}^{\infty} \text{supp }\frac{\mathds{1}_{\epsilon^{ns}\Om_n}}{\abs{\epsilon^{ns}\Om_n}} \right )- t_0\right ] \cap \Om_n = \left [\left (\sum_{s=1}^\infty \epsilon^{ns}\Om_n - t_0 \right )\right ] \cap \Om_n \neq  \emptyset\,.
\end{equation*}
Then, for all $n\geq N$,
\begin{align*}
\left |\psi_n(t_0) - \frac{\mathds{1}_{\Om}(t_0)}{\abs{\Om}} \right | &= \left | \left [\frac{\mathds{1}_{\Om_n}}{\abs{\Om_n}} \ast \left ( \Conv_{s=1}^{\infty} \, \frac{\mathds{1}_{\epsilon^{ns}\Om_n}}{\abs{\epsilon^{ns}\Om_n}}\right)\right ](t_0) - \frac{1}{\abs{\Om}} \right | \\
&= \left | \frac{1}{\abs{\Om_n}} \int_{\Om_n}  \left ( \Conv_{s=1}^{\infty} \, \frac{\mathds{1}_{\epsilon^{ns}\Om_n}}{\abs{\epsilon^{ns}\Om_n}}\right)(t_0-x)\, dx -  \frac{1}{\abs{\Om}} \right | \\
&\leq \left | \frac{1}{\abs{\Om_n}}\left \| \Conv_{s=1}^{\infty} \, \frac{\mathds{1}_{\epsilon^{ns}\Om_n}}{\abs{\epsilon^{ns}\Om_n}}\right \|_1 -  \frac{1}{\abs{\Om}} \right | \\
&= \left | \frac{1}{\abs{\Om_n}} -  \frac{1}{\abs{\Om}} \right | = \frac{\abs{\Om \setminus \Om_n}}{\abs{\Om_n}\abs{\Om}} \\
&< \frac{\delta}{(1-\epsilon^N)}\,,
\end{align*}
and claim $(i)$ follows.

For $(ii)$ it suffices to show that the sequence of Fourier transforms $\widehat{\psi_n}$ is uniformly Cauchy. By $(i)$ and Lemma~\ref{lemma:psi}, $\{\psi_n\}_{n=1}^{\infty}$ is a sequence of $L^1$ functions converging pointwise almost everywhere to $\frac{\mathds{1}_{\Om}}{\abs{\Om}} $.  Then by the $L^1$ Dominated Convergence Theorem, $\{\psi_n\}_{n=1}^{\infty}$ converges to $\frac{\mathds{1}_{\Om}}{\abs{\Om}} $ in $L^1$, and it follows that $\{\psi_n\}_{n=1}^{\infty}$  is an $L^1$ Cauchy sequence. Meaning, for $\delta>0$ there exists $N>0$ such that $\norm{\psi_n-\psi_m}_1 < \delta$ for all $n,m\geq N$.  Let $n,m\geq N$, then
\begin{equation*} 
\norm{\widehat{\psi_m} - \widehat{\psi_n}}_\infty \leq \norm{\psi_m - \psi_n}_1 < \delta\,,
\end{equation*}
and $\{\widehat{\psi_n}\}_{n=1}^{\infty}$ is uniformly Cauchy. By the completeness of $L^{\infty}(\R)$, it converges uniformly.
\end{proof}

Let, $\Psi$ denote the uniform limit of the sequence defined in \eqref{eq:function Fourier psi_n}. To make things more convenient later, we normalize $\Psi$, and define a new function
\begin{equation}\label{eq:function phi} 
\Phi = \abs{\Om} \cdot (\conv{\Psi}{\Psi})\,. 
\end{equation}
Note that $\Phi(0)=1$. 
 
The following observation will be the main ingredient in our approach. It is analogous to the results for regular frames of translates \cite{L02,HLW02}. With the above notation we have

\begin{proposition}\label{prop:limit function N}
Let $\La(\Om)$ be a simple generic model set with $\Om$ symmetric around the origin. Assume that $g_k,h_k\in W(\Lt,\ell^1)$ for every $k\in\mathcal{K}$, and $\sum_{k\in\mathcal{K}} \abs{\widehat{g_k}(t)}^2 \leq B_g$, $\sum_{k\in\mathcal{K}} \abs{\widehat{h_k}(t)}^2 \leq B_h$, with $B_g,B_h>0$. Then, for every $f\in\mathcal{D}$, the function
\begin{equation*}
\mathcal{N}(x) = \sum_{k\in\mathcal{K}} \sum_{\la\in\La(\Om)} \inner{T_x f}{T_\la g_k} \inner{T_\la h_k}{T_x f}
\end{equation*}
is continuous and coincides pointwise with its Fourier series $\sum_{\eta\in\Ga^*} \widehat{\mathcal{N}}(\eta) e^{-2\pi i p_1^*(\eta) \cdot x}$, with
\begin{equation*}
\widehat{\mathcal{N}}(\eta) = D(\La(\Om)) \, \Phi(-p_2^*(\eta)) \int_{\R^m} \widehat{f}(t) \overline{\widehat{f}(t-p^*_1(\eta))} \left (\sum_{k\in\mathcal{K}} \overline{\widehat{g_k}(t)}\,\widehat{h_k}(t-p^*_1(\eta))\right ) \, dt \,.
\end{equation*}
where $\eta\in \Ga^\ast$, $\Phi$ defined in \eqref{eq:function phi}, and the integral converges absolutely.
\end{proposition}

\begin{remark}
Let $W$ be the essential support of $\Phi$. Then, for $p_2^*(\eta) \notin W$, the coefficients $\widehat{\mathcal{N}}(\eta)$ are essentially zero, and we can approximate $\mathcal{N}(x)$ by 
\begin{equation*}
\mathcal{N}(x) \approx \sum_{\be\in\La(W)} \widehat{\mathcal{N}}(\be) e^{-2\pi i \be \cdot x}\,,
\end{equation*}
where $\La(W) = \left \{\be=p_1^*(\eta)\,:\, \eta\in\Ga^*,\, p_2^*(\eta) \in W \right \}$ is a model set originating from $\Ga^*$ and 
\begin{equation*}
\widehat{\mathcal{N}}(\be) = D(\La(\Om)) \, \Phi(-p_2^*(\eta)) \int_{\R^m} \widehat{f}(t) \overline{\widehat{f}(t-\be)} \left (\sum_{k\in\mathcal{K}} \overline{\widehat{g_k}(t)}\, \widehat{h_k}(t-\be)\right ) \, dt \,.
\end{equation*}
\end{remark}

\begin{proof}[Proof of Proposition~\ref{prop:limit function N}]
Let $f\in\mathcal{D}$. Let $\psi_n$ be a  sequence of $C_0^\infty(\Om)$ non-negative functions defined in \eqref{eq:function psi_n}. By Proposition~\ref{prop:function N}, for each $n\in\N$, the functions
\begin{equation*}
\mathcal{N}^{\psi_n}(x) = \sum_{k\in\mathcal{K}} \sum_{\la\in\La(\Om)} w_{\psi_n}(\la)^2\, \inner{T_x f}{T_\la g_k} \inner{T_\la h_k}{T_x f}
\end{equation*}
are well defined almost periodic functions that are pointwise equal to their Fourier series $\sum_{\eta\in\Ga^*} \widehat{\mathcal{N}^{\psi_n}}(\eta) e^{-2\pi i p_1^*(\eta) \cdot x}$, where $\widehat{\mathcal{N}^{\psi_n}}(\eta)$ are given by
\begin{equation*}
\widehat{\mathcal{N}^\psi_n}(\eta) = \text{vol}(\Ga)^{-1}\,\widehat{\psi_n^2} (-p_2^*(\eta)) \int_{\R^m} \widehat{f}(t) \overline{\widehat{f}(t-p^*_1(\eta))} \left (\sum_{k\in\mathcal{K}} \overline{\widehat{g_k}(t)}\,\widehat{h_k}(t-p^*_1(\eta))\right ) \, dt \,.
\end{equation*}

By Cauchy-Schwarz inequality, we have
\begin{equation*}
\absBig{\sum_{k\in\mathcal{K}}  \overline{\widehat{g_k}(t)}\, \widehat{h_k}(t-p_1^*(\eta))} \leq  \left ( \sum_{k\in\mathcal{K}} \absbig{\widehat{g_k}(t)}^2\right )^{1/2} \left ( \sum_{k\in\mathcal{K}} \absbig{\widehat{h_k}(t-p_1^*(\eta))}^2\right )^{1/2} \leq B_g^{1/2} B_h^{1/2} \,,
\end{equation*}
and the function $\sum_{k\in\mathcal{K}}  \overline{\widehat{g_k}(t)} \, \widehat{h_k}(t-p_1^*(\eta))$ is locally integrable, and the coefficients $\widehat{\mathcal{N}}(\eta)$ are well defined. The series $\sum_{\eta\in\Ga^*} \widehat{\mathcal{N}}(\eta) e^{-2\pi i p_1^*(\eta) \cdot x}$ converges absolutely (by a similar argument as in the proof of Proposition~\ref{prop:ap bracket}) and gives rise to a uniformly continuous function. By the uniform convergence of $\widehat{\psi_n^2}$ to $\conv{\Psi}{\Psi}$, it can be easily verified that $\mathcal{N}^{\psi_n}$ converges uniformly to $\abs{\Om}^{-2}\,\sum_{\eta\in\Ga^*} \widehat{\mathcal{N}}(\eta) e^{-2\pi i p_1^*(\eta) \cdot x}$, since $D(\La(\Om)) = \text{vol}(\Ga)^{-1} \abs{\Om}$ and $\Phi= \abs{\Om} \cdot (\conv{\Psi}{\Psi})$.

Now, since $\La(\Om)$ is generic, that is the boundary $\partial \Om$ of $\Om$ has no common points with $p_2(\Ga)$,  and $\psi_n$ converges pointwise to $\frac{\mathds{1}_\Om}{\abs{\Om}}$ on $\Om\setminus \partial \Om$, by Lemma~\ref{lemma:psi}, we show that $\mathcal{N}(x)$ is a pointwise limit of $\abs{\Om}^2 \,\mathcal{N}^{\psi_n}(x)$. Indeed, by the Lebesgue Dominated Convergence Theorem, we can move the limit inside the sum, and for every $x\in\R^m$, we have
\begin{align*}
\lim_{n\rightarrow \infty} \mathcal{N}^{\psi_n}(x) &=  \sum_{k\in\mathcal{K}} \sum_{\la\in\La(\Om)} \lim_{n\rightarrow \infty} w_{\psi_n}^2(\la) \inner{T_x f}{T_\la g_k} \inner{T_\la h_k}{T_x f}\\
& = \sum_{k\in\mathcal{K}} \sum_{\la\in\La(\Om)} \lim_{n\rightarrow \infty} \psi_n^2(p_2(\ga)) \, \inner{T_x f}{T_\la g_k}\inner{T_\la h_k}{T_x f}\quad (\la=p_1(\ga),\, \ga\in\Ga)\\
&= \sum_{k\in\mathcal{K}} \sum_{\la\in\La(\Om)} \abs{\Om}^{-2}\big [ \mathds{1}_{\Om}(p_2(\ga))\big ]^2 \, \inner{T_x f}{T_\la g_k}\inner{T_\la h_k}{T_x f}\\
& =  \abs{\Om}^{-2} \mathcal{N}(x)\,.
\end{align*}
By the uniqueness of the limits, we must have $\mathcal{N}(x) = \sum_{\eta\in\Ga^*} \widehat{\mathcal{N}}(\eta) e^{-2\pi i p_1^*(\eta) \cdot x}$, and by the uniqueness of the Fourier series, \eqref{eq:series} is the Fourier series of $\mathcal{N}(x)$.
\end{proof}

Assuming that $\{T_\la g_k:k\in\mathcal{K}, \la\in\La(\Om)\}$ and $\{T_\la h_k:k\in\mathcal{K}, \la\in\La(\Om)\}$ are Bessel sequences, it follows from Proposition~\ref{prop:bessel}, that the sums $\sum_{k\in\mathcal{K}}\abs{\widehat{g_k}(t)}^2$ and $\sum_{k\in\mathcal{K}}\abs{\widehat{h_k}(t)}^2$ are bounded from above and the function $\mathcal{N}(x)$ of Proposition~\ref{prop:limit function N} coincides with the function $\mathcal{N}(x)$ defined in \eqref{eq:main function}. Indeed, using Proposition~\ref{prop:operator shifts}, we can write $\mathcal{N}(x)$ from \eqref{eq:main function} explicitly as
\begin{align*}
\mathcal{N}(x) &=  \widetilde{\mathcal{N}}(\La(\Om)-x) = \inner{S_{g,h}^{\La(\Om)-x}f}{f} = \inner{S_{g,h}^{\La(\Om)}T_x f}{T_x f}\\ 
&= \sum_{k\in\mathcal{K}} \sum_{\la\in\La(\Om)} \inner{T_x f}{T_\la g_k} \inner{T_\la h_k}{T_x f}\,.
\end{align*}
By the uniqueness of the Fourier coefficients, $\widehat{\mathfrak{N}}(\eta)$ in \eqref{eq:fourier coeff} equal $\widehat{\mathcal{N}}(\eta)$ from Proposition~\ref{prop:limit function N}, for all $\eta\in\Ga^*$.

We are now in a position to characterize tight frames and dual frames for $\Lt(\R^m)$. We start with tight frames. The first result is analogous to the known result for regular frames of translates which is due to Ron and Shen \cite{RS95} (see also \cite{J98}) and was reformulated with a different proof in \cite{HLW02}. Our results for model sets relies on the similar methods developed in \cite{HLW02}.

\begin{theorem}\label{thm:tight frame1}
Let $\La(\Om)$ be a simple generic model set, $\Om$ symmetric around the origin, and $g_k\in W(\Lt,\ell^1)$ for every $k\in\mathcal{K}$. Then the family $\{T_\la g_k:k\in\mathcal{K}, \la\in\La(\Om)\}$ is a normalized tight frame for $\Lt(\R^m)$, that is 
\begin{equation}\label{eq:tight1}
\sum_{k\in\mathcal{K}} \sum_{\la\in\La(\Om)} \abs{\inner{f}{T_\la g_k}}^2 = \norm{f}_2^2 \quad \mbox{for all $f\in\Lt(\R^m)$}
\end{equation} 
if and only if
\begin{equation}\label{eq:cond1}
D(\La(\Om))\, \Phi (-p_2^*(\eta)) \sum_{k\in\mathcal{K}}\overline{\widehat{g_k}(t)}\,\widehat{g_k}(t-p_1^*(\eta)) = \delta_{\eta,0} \quad \mbox{for a.e. $t\in\R^m$}\,,
\end{equation}
for each $\eta\in \Ga^*$, where $\delta$ is the Kronecker delta and $\Phi$ is defined in \eqref{eq:function phi}.
\end{theorem}

\begin{proof}
It is sufficient to prove the theorem when $f\in\mathcal{D}$ \cite{HW96}. If relation \eqref{eq:tight1} holds, then $\{T_\la g_k:k\in\mathcal{K},\la\in\La(\Om)\}$ is a Bessel sequence and by the Remark after Proposition~\ref{prop:bessel}, $\sum_{k\in\mathcal{K}}\abs{\widehat{g_k}(t)}^2$ is bounded from above. Also, if  \eqref{eq:tight1} holds, then by Proposition~\ref{prop:operator shifts}, $\{T_\la g_k:k\in\mathcal{K},\la\in\La(\Om)-x\}$ is a tight frame for all $x\in\R^m$ and, so the function
\begin{equation}\label{eq:function O}
\widetilde{\mathcal{O}}(\La(\Om)-x) := \sum_{k\in\mathcal{K}} \sum_{\la\in\La(\Om)-x} \abs{\inner{f}{T_\la g_k}}^2
\end{equation}
is constant on the orbit of $\La(\Om)$ and equals $\norm{f}_2^2$. We can define the function $\mathcal{O}(x)$ on all of $\R^m$ as in \eqref{eq:main function}. It coincides with the function $\mathcal{N}(x)$ of Proposition~\ref{prop:limit function N}, when $g_k=h_k$ for every $k\in\mathcal{K}$. Then, $\mathcal{E}(x) = \mathcal{O}(x) - \norm{f}_2^2$, for all $f\in\mathcal{D}$, is given by the trigonometric series with coefficients
\begin{equation*}
\widehat{\mathcal{E}}(\eta) = \left \{ \begin{array}{ll} \widehat{\mathcal{O}}(0) - \norm{f}_2^2\,, &\eta = 0\\
\widehat{\mathcal{O}}(\eta)\,, &\eta \neq 0\,, \end{array} \right .
\end{equation*}
for all $\eta\in\Ga^*$. 
Therefore, if \eqref{eq:tight1} holds, then $\mathcal{E}(x)=0$, and all its coefficients have to be zero. Thus, by Proposition~\ref{prop:limit function N} with $g_k=h_k$ for every $k\in\mathcal{K}$, we have 
\begin{equation}\label{eq:identity1}
D(\La(\Om))\, \Phi(-p_2^*(\eta)) \int_{\R^m} \widehat{f}(t) \overline{\widehat{f}(t-p_1^*(\eta))} \left (\sum_{k\in\mathcal{K}} \overline{\widehat{g_k}(t)}\,\widehat{g_k}(t-p_1^*(\eta)) \right )\, dt = \delta_{\eta,0} \norm{f}_2^2 \,,
\end{equation}
with $\eta\in \Ga^*$, for every $f\in \mathcal{D}$. 

Consider the case $\eta=0$ and let 
\begin{equation*}
G_0(t) = D(\La(\Om))\, \Phi(0) \sum_{k\in\mathcal{K}} \abs{\widehat{g_k}(t)}^2 = D(\La(\Om))\, \sum_{k\in\mathcal{K}} \abs{\widehat{g_k}(t)}^2 \,.
\end{equation*}
From the Bessel property of $\{T_\la g_k:k\in\mathcal{K},\la\in\La(\Om)\}$, it follows that
\begin{equation}\label{eq:bound on g}
D(\La(\Om))\, \sum_{k\in\mathcal{K}} \abs{\widehat{g_k}(t)}^2 \leq 1 \quad \mbox{for a.e. $t\in\R^m$}\,,
\end{equation}
and $G_0$ is locally integrable. Let $s$ be a bounded compactly supported function given by $s(t) = \mathds{1}_{B(0,1)}(t)$, where $B(0,1) = [0,1]^m$ is a ball in $\R^m$ centered at zero. We form a sequence of functions $s_n(t) = \sqrt{n} \,s(n\cdot t) = \sqrt{n}\, \mathds{1}_{B(0,1/n)}(t)$. Then $\{ s_n^2\}_{n\in\N}$ is an approximate identity and for each Lebesgue point $t_0\in\R^m$ of $G_0$, we have
\begin{align}
\lim_{n\rightarrow \infty}\big (s_n^2 \ast G_0\big )(t_0) &= \lim_{n\rightarrow \infty}\int_{\R^m} s_n^2(t_0-t) G_0(t)\, dt \nonumber \\ 
&=  \lim_{n\rightarrow \infty} n^m  \int_{\abs{t_0-t}\leq (1/n)^m} G_0(t)\,dt = G_0(t_0)\,. \label{eq:aprox identity1}
\end{align}
Therefore, by letting $\widehat{f_n}(t)=s_n(t_0-t)$, we have $\norm{f_n}_2=1$ and $f_n\in\mathcal{D}$. By \eqref{eq:identity1} and \eqref{eq:aprox identity1},
\begin{equation*}
1 =  \norm{f_n}_2^2 =  \lim_{n\rightarrow \infty}  \int_{\R^m} \absbig{\widehat{f_n}(t)}^2 G_0(t)\,dt = G_0(t_0) \quad \mbox{for a.e. $t_0\in\R^m$,}
\end{equation*}
and \eqref{eq:cond1} is satisfied for $\eta=0$. When $\eta\neq 0$, let
\begin{equation*}
G_\eta(t) = D(\La(\Om))\,\Phi(-p_2^*(\eta)) \sum_{k\in\mathcal{K}} \overline{\widehat{g_k}(t)}\, \widehat{g_k}(t-p_1^*(\eta))\,.
\end{equation*}
By polarization of \eqref{eq:identity1} we have
\begin{equation*}
\int_{\R^m} \widehat{f^1}(t) \overline{\widehat{f^2}(t-p_1^*(\eta))} G_\eta(t) = 0
\end{equation*}
for all $f^1,f^2\in\mathcal{D}$. By Cauchy-Schwarz inequality and \eqref{eq:bound on g}, $G_\eta$ is locally integrable. We choose $f^1_n$ and $f^2_n$ such that
\begin{equation*}
\widehat{f^1_n}(t) = s_n(t_0-t) \quad \mbox{and} \quad \widehat{f^2_n}(t) = s_n(t_0-t+p_1^*(\eta))\,.
\end{equation*}
Then $\norm{f^1_n}_2=1$ and $\norm{f^2_n}_2=1$, and for each Lebesgue point $t_0\in\R^m$ of $G_\eta$, we have
\begin{align*}
0&=\lim_{n\rightarrow \infty}\int_{\R^m} \widehat{f^1_n}(t) \overline{\widehat{f^2_n}(t-p_1^*(\eta))}  G_\eta(t)\, dt= \lim_{n\rightarrow \infty}\int_{\R^m} s_n^2(t_0-t) G_\eta(t)\, dt\\
&= \lim_{n\rightarrow \infty}\big (s_n^2 \ast G_\eta \big )(t_0) = G_\eta(t_0)\,.
\end{align*}
Hence, \eqref{eq:cond1} is satisfied for $\eta\neq 0$.

Conversely, assume that \eqref{eq:cond1} holds. Then $D(\La(\Om))\,\sum_{k\in\mathcal{K}} \abs{\widehat{g_k}(t)}^2 = 1$ for a.e. $t\in\R^m$, and, by Proposition~\ref{prop:limit function N} with $g_k=h_k$ for every $k\in\mathcal{K}$, we can define function $\mathcal{O}(x)$ as in \eqref{eq:function O}. Then, by Proposition~\ref{prop:limit function N} and relation \eqref{eq:cond1}, the function $\mathcal{O}(x)$ is given by the trigonometric series
\begin{equation*}
\sum_{\eta\in \Ga^*} \widehat{\mathcal{O}}(\eta) e^{-2\pi i x \cdot p_1^*(\eta)} = \mathcal{O}(x)\,,
\end{equation*}
for every $x\in\R^m$, with
\begin{equation*}
\widehat{\mathcal{O}}(\eta) = \delta_{\eta,0} \int_{\R^m} \widehat{f}(t) \overline{\widehat{f}(t-p_1^*(\eta))}\, dt\,.
\end{equation*}
Therefore, $\mathcal{O}(x)$ is constant and $\mathcal{O}(x) = \norm{f}_2^2$. This means that $\{T_\la g_k: k\in\mathcal{K}, \la\in\La(\Om)-x\}$ is a tight frame for all $x\in\R^m$, in particular, for $x=0$, which proves the claim.
\end{proof}

We can also charcterize tight frames in a different manner. The following result is similar to the characterization obtained by Labate et. al.. in \cite{L02, HLW02}  for the family of tight frame generators under multi-integer shifts.

\begin{theorem}\label{thm:tight frame2}
Let $\La(\Om)$ be a simple generic model set with $\Om$ symmetric around the origin and $g_k\in W(\Lt,\ell^1)$ for every $k\in\mathcal{K}$. Then the family $\{T_\la g_k: k\in\mathcal{K}, \la\in\La(\Om)\}$ is a normalized tight frame for $\Lt(\R^m)$ if and only if it is a  Bessel sequence with upper frame bound $B_g=1$ and 
\begin{equation}\label{eq:rel2}
D(\La(\Om))\, \sum_{k\in\mathcal{K}} \absbig{\widehat{g_k}(t)}^2 = 1 \quad \mbox{for a.e. $t\in\R^m$}
\end{equation}
holds.
\end{theorem}

\begin{proof}
Let us first assume that $\{T_\la g_k:k\in\mathcal{K}, \la\in\La(\Om)\}$ is a normalized tight frame for $\Lt(\R^m)$. Then the Bessel property follows from the tight frame property of $\{T_\la g_k:k\in\mathcal{K}, \la\in\La(\Om)\}$, and relation \eqref{eq:rel2} follows from Theorem~\ref{thm:tight frame1}.

For the converse, by $\{T_\la g_k:k\in\mathcal{K}, \la\in\La(\Om)\}$ being a Bessel sequence with upper frame bound $B_{g}=1$, and by the remark after Proposition~\ref{prop:bessel}, $\sum_{k\in\mathcal{K}} \abs{\widehat{g_k}(t)}^2$ is bounded from above for a.e. $t\in\R^m$. Hence, by Proposition~\ref{prop:limit function N} with $g_k=h_k$ for every $k\in\mathcal{K}$, the function $\mathcal{E}(x) = \norm{f}_2^2 - \mathcal{O}(x)$ is almost periodic, where $\mathcal{O}(x) = \sum_{k\in\mathcal{K}} \sum_{\la\in\La(\Om)} \abs{\inner{T_x f}{T_\la g_k}}^2$.
We then have
\begin{align*}
\M_x\big\{ \mathcal{E}(x)\big \} &= \norm{f}_2^2 - \M_x\big\{ \mathcal{O}(x)\big \}  = \norm{f}_2^2 - \widehat{\mathcal{O}}(0)\\ 
&= \norm{f}_2^2 - \text{vol}(\Ga)^{-1} \int_{\R^m} \absbig{\widehat{f}(t)}^2 \sum_{k\in\mathcal{K}}\absbig{\widehat{g_k}(t)}^2 \, dt\,,
\end{align*}
where we used Proposition~\ref{prop:limit function N} to compute $\widehat{\mathcal{O}}(0)$. By \eqref{eq:rel2}, $\widehat{\mathcal{O}}(0) = \norm{f}_2^2$, and  therefore $\M_x\{ \mathcal{E}(x)\}=0$. Since, by the Bessel property of $\{T_\la g_k:k\in\mathcal{K}, \la\in\La(\Om)\}$, $\mathcal{E}(x) \geq 0$, we can apply Theorem~\ref{thm:ap non-negative} and conclude that $\mathcal{O}(x) = \norm{f}_2^2$.
This means that $\{T_\la g_k: k\in\mathcal{K}, \la\in\La(\Om)-x\}$ is a normalized tight frame for all $x\in\R^m$, in particular, for $x=0$, which proves the claim.
\end{proof}

From Proposition~\ref{prop:limit function N}, we will now obtain a characterization of the dual frame generators for $\Lt(\R^m)$. A characterization of the shift-invariant dual frame generators was done in \cite{L02,HLW02} and similar one in \cite{RS95}.

\begin{theorem}\label{thm:dual frame}
Let $\La(\Om)$ be a simple generic model set, $\Om$ symmetric around the origin and $g_k, h_k \in W(\Lt,\ell^1)$ for every $k\in\mathcal{K}$. Let $\{T_\la g_k: k\in\mathcal{K}, \la\in\La(\Om)\}$ and $\{T_\la h_k:k\in\mathcal{K}, \la\in\La(\Om)\}$ be Bessel sequences in $\Lt(\R^m)$, with upper frame bounds $B_{g}$ and $B_{h}$, respectively. Then
\begin{equation}\label{eq:tight}
\sum_{k\in\mathcal{K}} \sum_{\la\in\La(\Om)} \inner{f}{T_\la g_k}  \inner{T_\la h_k}{f} = \norm{f}_2^2 \quad \mbox{for all $f\in\Lt(\R^m)$}
\end{equation} 
if and only if
\begin{equation}\label{eq:cond}
D(\La(\Om))\, \Phi(-p_2^*(\eta)) \sum_{k\in\mathcal{K}} \overline{\widehat{g_k}(t)}\, \widehat{h_k}(t-p_1^*(\eta)) = \delta_{\eta,0} \quad \mbox{for a.e. $t\in\R^m$}\,,
\end{equation}
for each $\eta\in \Ga^*$, where $\delta$ is the Kronecker delta and $\Phi$ is defined in \eqref{eq:function phi}.
\end{theorem}

\begin{proof}
It suffices to prove the theorem for $f\in\mathcal{D}$ \cite{FGWW97}. Since $\{T_\la g_k:k\in\mathcal{K}, \la\in\La(\Om)\}$ and $\{T_\la h_k:k\in\mathcal{K}, \la\in\La(\Om), k\in\mathcal{K}\}$ are Bessel sequences, it follows from the remark after Proposition~\ref{prop:bessel} that
\begin{equation}\label{eq:g and h}
D(\La(\Om))\, \sum_{k\in\mathcal{K}} \absbig{\widehat{g_k}(t)}^2 \leq B_{g} \quad \mbox{and} \quad D(\La(\Om))\, \sum_{k\in\mathcal{K}} \absbig{\widehat{h_k}(t)}^2 \leq B_{h}
\end{equation}
for a.e. $t\in\R^m$. Therefore we can apply Proposition~\ref{prop:limit function N}, which involves the function $\mathcal{N}(x)$. If \eqref{eq:tight} holds, then $\mathcal{N}(x) = \norm{f}_2^2$ for all $x\in\R^m$. Since $\mathcal{N}(x)$ is almost periodic, so is $\mathcal{E}(x) = \mathcal{N}(x)-\norm{f}_2^2$ with Fourier coefficients given by
\begin{equation*}
\widehat{\mathcal{E}}(\eta) = \left \{ \begin{array}{ll} \widehat{\mathcal{N}(0)}- \norm{f}_2^2\,, &\eta = 0\\
\widehat{\mathcal{N}}(\eta)\,, &\eta \neq 0\,, \end{array} \right .
\end{equation*}
Since $\mathcal{E}(x)=0$, by the uniqueness of Fourier coefficients of almost periodic functions, all its Fourier coefficients have to be zero. Thus, by Proposition~\ref{prop:limit function N}, we have 
\begin{equation}\label{eq:identity2}
D(\La(\Om))\, \Phi(-p_2^*(\eta)) \int_{\R^m} \widehat{f}(t) \overline{\widehat{f}(t-p_1^*(\eta))} \left (\sum_{k\in\mathcal{K}} \overline{\widehat{g_k}(t)}\,\widehat{h_k}(t-p_1^*(\eta)) \right )\, dt = \delta_{\eta,0} \norm{f}_2^2 \,,
\end{equation}
with $\eta\in \Ga^*$, for every $f\in \mathcal{D}$. 
Analogously to the proof of Theorem~\ref{thm:tight frame1}, consider the case $\eta=0$ and let 
\begin{equation*}
H_0(t) = D(\La(\Om))\, \sum_{k\in\mathcal{K}} \overline{\widehat{g_k}(t)}\, \widehat{h_k}(t)\,.
\end{equation*}
By the Cauchy-Schwarz inequality and \eqref{eq:g and h}, $H_0$ is locally integrable. Let $s$ be a bounded compactly supported function given by $s(t) = \mathds{1}_{B(0,1)}(t)$, where $B(0,1) = [0,1]^m$ is a ball in $\R^m$ centered at zero. We form a sequence of functions $s_n(t) = \sqrt{n} \,s(n\cdot t) = \sqrt{n}\, \mathds{1}_{B(0,1/n)}(t)$. Then $\{ s_n^2\}_{n\in\N}$ is an approximate identity and for each Lebesgue point $t_0\in\R^m$ of $H_0$, we have
\begin{align}
\lim_{n\rightarrow \infty}\big (s_n^2 \ast H_0\big )(t_0) &= \lim_{n\rightarrow \infty}\int_{\R^m} s_n^2(t_0-t) H_0(t)\, dt \nonumber \\ 
&=  \lim_{n\rightarrow \infty} n^m  \int_{\abs{t_0-t}\leq (1/n)^m} H_0(t)\,dt = H_0(t_0)\,. \label{eq:aprox identity2}
\end{align}
Therefore, by letting $\widehat{f_n}(t)=s_n(t_0-t)$, we have $\norm{f_n}_2=1$ and $f_n\in\mathcal{D}$. By \eqref{eq:identity2} and \eqref{eq:aprox identity2},
\begin{equation*}
1 =  \norm{f_n}_2^2 =  \lim_{n\rightarrow \infty}  \int_{\R^m} \absbig{\widehat{f_n}(t)}^2 H_0(t)\,dt = H_0(t_0) \quad \mbox{for a.e. $t_0\in\R^m$,}
\end{equation*}
and \eqref{eq:cond} is satisfied for $\eta=0$. When $\eta\neq 0$, let
\begin{equation*}
H_\eta(t) = D(\La(\Om))\, \Phi(-p_2^*(\eta)) \sum_{k\in\mathcal{K}} \overline{\widehat{g_k}(t)}\, \widehat{h_k}(t-p_1^*(\eta))\,.
\end{equation*}
By polarization of \eqref{eq:identity2} we have
\begin{equation*}
\int_{\R^m} \widehat{f^1}(t) \overline{\widehat{f^2}(t-p_1^*(\eta))} H_\eta(t) = 0
\end{equation*}
for all $f^1,f^2\in\mathcal{D}$. By Cauchy-Schwarz inequality and \eqref{eq:g and h}, $H_\eta$ is locally integrable. We choose $f^1_n$ and $f^2_n$ such that
\begin{equation*}
\widehat{f^1_n}(t) = s_n(t_0-t) \quad \mbox{and} \quad \widehat{f^2_n}(t) = s_n(t_0-t+p_1^*(\eta))\,.
\end{equation*}
Then $\norm{f^1_n}_2=1$ and $\norm{f^2_n}_2=1$, and for each Lebesgue point $t_0\in\R^m$ of $H_\eta$, we have
\begin{align*}
0&=\lim_{n\rightarrow \infty}\int_{\R^m} \widehat{f^1_n}(t) \overline{\widehat{f^2_n}(t-p_1^*(\eta))}  H_\eta(t)\, dt= \lim_{n\rightarrow \infty}\int_{\R^m} s_n^2(t_0-t) H_\eta(t)\, dt\\
&= \lim_{n\rightarrow \infty}\big (s_n^2 \ast H_\eta \big )(t_0) = H_\eta(t_0)\,.
\end{align*}
Hence, \eqref{eq:cond} is satisfied for $\eta\neq 0$.

Conversely, assume that \eqref{eq:cond} holds. Since $\{T_\la g_k: k\in\mathcal{K}, \la\in\La(\Om)\}$ and $\{T_\la h_k:k\in\mathcal{K}, \la\in\La(\Om)\}$ be Bessel sequences in $\Lt(\R^m)$, the function $\mathcal{N}(x)$, as in Proposition~\ref{prop:limit function N}, is well defined. It coincides pointwise with its Fourier series 
\begin{equation*}
\mathcal{N}(x) = \sum_{\eta \in \Ga^*} \widehat{\mathcal{N}}(\eta) e^{-2\pi i x \cdot p_1^*(\eta)}\,,
\end{equation*}
for all $x\in\R^m$, and by \eqref{eq:cond}
\begin{equation*}
\widehat{\mathcal{N}}(\eta) = \delta_{\eta,0} \int_{\R^m} \widehat{f}(t) \overline{\widehat{f}(t-p_1^*(\eta))}\, dt\,.
\end{equation*}
Therefore, $\mathcal{N}$ is a constant function and $\mathcal{N}(x) = \norm{f}_2^2$ for all $x\in\R^m$. That means that $\{T_\la g_k: k\in\mathcal{K}, \la\in\La(\Om)-x\}$ and $\{T_\la h_k:k\in\mathcal{K}, \la\in\La(\Om)-x\}$ are  dual frames, for all $x\in\R^m$. Taking $x=0$ proves the claim.
\end{proof}


\section{Gabor Frames}\label{sec:gabor}

In this section we apply the developed results to Gabor frames. Let $\La(\Om)$ be a  simple generic model set, $\Om$ symmetric around the origin, and $\mathcal{K}=\mathcal{L}\times \Delta$, where $\mathcal{L}=\{1,\ldots,L\}$ is a finite index set and $\Delta$ a relatively dense and uniformly discrete set in $\R^m$. Let $g_k:= M_{\be} g_l$ for $k=(l,\beta)\in\mathcal{K}$ and $g_l\in W(\Lt,\ell^1)$. Then the system $\{ T_\la g_k:k\in\mathcal{K}, \la\in\La(\Om)\}$ is the Gabor system
\begin{equation}\label{eq:Gabor}
\mathcal{G}(g,\La(\Om),\Delta): = \{ T_\la M_{\be} g_l: l=1,\ldots, L, \la\in\La(\Om),\be\in\Delta\}\,.
\end{equation}

The assumption in Proposition~\ref{prop:limit function N}, that $\sum_{k\in\mathcal{K}}\abs{\widehat{g_k}(t)}^2$ should be bounded from above, which follows from the Bessel property of $\{ T_\la g_k:k\in\mathcal{K}, \la\in\La(\Om)\}$, is automatically satisfied for Gabor frames. Indeed, if $g_k =  M_{\be} g_l$ and $g_l\in W(\Lt,\ell^1)$, then
\begin{equation*}
\sum_{k\in\mathcal{K}}\abs{\widehat{g_k}(t)}^2 = \sum_{l=1}^L \sum_{\be\in\Delta}\abs{\widehat{g_l}(t-\be)}^2 \leq \sum_{l=1}^L \text{rel}(\Delta) \norm{\widehat{g_l}}_{\infty,2}^2 \leq L\, \text{rel}(\Delta) \, \max_{l=1,\ldots,L} \norm{\widehat{g_l}}_{\infty,2}^2
\end{equation*}
since $\widehat{g_l}\in W(L^\infty,\ell^2)$, and $\text{rel}(\Delta)$ is the relative separation of $\Delta$.

The following characterization of tight Gabor frames is a simple corollary of Theorem~\ref{thm:tight frame1}. In the case of $\La(\Om)$ being a lattice, they were obtained in \cite{RS97, CCJ99, Cz00, HLW02}. 

\begin{corollary}\label{col:tight gabor 1}
The system $\mathcal{G}(g,\La(\Om),\Delta)$ is a normalized tight Gabor frame for $\Lt(\R^m)$ if and only if
\begin{equation*}
D(\La(\Om))\, \Phi(-p_2^*(\eta)) \sum_{l=1}^L \sum_{\be\in\Delta}\overline{\widehat{g_l}(t-\beta)}\,\widehat{g_l}(t-\beta - p_1^*(\eta)) = \delta_{\eta,0} \quad \mbox{for a.e. $t\in\R^m$}\,,
\end{equation*}
for each $\eta\in \Ga^*$, where $\delta$ is the Kronecker delta and $\Phi$ defined in \eqref{eq:function phi}.
\end{corollary}

\begin{proof}
We apply Theorem~\ref{thm:tight frame1} with $g_k:= M_{\be} g_l$. Observe that $\widehat{g_k}(t)  = \widehat{g_l}(t-\beta)$ and the sum in the left hand side of \eqref{eq:cond1} becomes
\begin{equation*}
\sum_{k\in\mathcal{K}} \overline{\widehat{g_k}(t)}\, \widehat{g_k}(t-p_1^*(\eta)) = \sum_{l=1}^L \sum_{\be\in\Delta}\overline{\widehat{g_l}(t-\beta)}\, \widehat{g_l}(t-\beta - p_1^*(\eta))\,.
\end{equation*}
\end{proof}

For the Gabor systems $\mathcal{G}(g,\La(\Om),\Delta)$, condition \eqref{eq:rel2} of Theorem~\ref{thm:tight frame2} 
becomes 
\begin{equation}\label{eq:calderon}
D(\La(\Om))\, \sum_{l=1}^L \sum_{\be\in\Delta} \abs{\widehat{g_l}(t-\beta)}^2 = 1 \quad \mbox{for a.e. $t\in\R^m$}\,,
\end{equation}
and we obtain another characterization of tight Gabor frames 
\begin{corollary}
The system $\mathcal{G}(g,\La(\Om),\Delta)$ is a normalized tight Gabor frame for $\Lt(\R^m)$ if and only if it is a Bessel sequence with upper frame bound $B_g=1$ and \eqref{eq:calderon} holds.
\end{corollary}

The application of Theorem~\ref{thm:dual frame} to the Gabor system $\mathcal{G}(g,\La(\Om),\Delta)$ yields the following characterization of Gabor dual frames. 

\begin{corollary}\label{col:dual frames}
Let $h_l\in W(\Lt,\ell^1)$ for each $l\in\mathcal{L}$, and assume that $\mathcal{G}(g,\La(\Om),\Delta)$ and $\mathcal{G}(h,\La(\Om),\Delta)$ are Bessel sequences for $\Lt(\R^m)$. Then $\mathcal{G}(h,\La(\Om),\Delta)$ is a dual system to $\mathcal{G}(g,\La(\Om),\Delta)$ if and only if
\begin{equation*}
D(\La(\Om))\, \Phi(-p_2^*(\eta)) \sum_{l=1}^L \sum_{\be\in\Delta}\overline{\widehat{g_l}(t-\beta)}\, \widehat{h_l}(t-\beta - p_1^*(\eta)) = \delta_{\eta,0} \quad \mbox{for a.e. $t\in\R^m$}\,,
\end{equation*}
for each $\eta\in \Ga^*$, where $\delta$ is the Kronecker delta and $\Phi$ is defined in \eqref{eq:function phi}.
\end{corollary}                                                                                                                                               

We note here, that if $\Delta$ is a lattice in $\R^m$, that is $\Delta = A\Z^m$, where $A\in GL_m(\R)$, then we can deduce a form of Wexler-Raz biorthogonality relations for Gabor systems of the form $\mathcal{G}(g,\La(\Om),A\Z^m)$, that is when translations are taken from model sets and modulations from a lattice. An analogous theorem for Gabor system on a lattice was found by Wexler and Raz \cite{WR90} and proved in \cite{J94,J95,DLL95,RS95}, as well as in \cite{HLW02}.

\begin{theorem}\label{thm:wexler raz}
Assume that $\mathcal{G}(g,\La(\Om), A\Z^m)$ and $\mathcal{G}(h,\La(\Om), A\Z^m)$ are Bessel sequences for $\Lt(\R^m)$. Then  $\mathcal{G}(h,\La(\Om),A\Z^m)$ is a dual system to $\mathcal{G}(g,\La(\Om),A\Z^m)$ if and only if
\begin{equation}\label{eq:W-R}
D(\La(\Om))\, \abs{\det A}^{-1} \, \Phi(-p_2^*(\eta))  \sum_{l=1}^L  \innerBig{T_{A^I v}\,M_{p_1^*(\eta)} \,h_l}{g_l}  = \delta_{\eta,0}\, \delta_{v,0}\,,
\end{equation}
for each $\eta\in \Ga^*$ and $v\in\Z^m$, where $A^I = (A^T)^{-1}$ and $\delta$ is the Kronecker delta and $\Phi$ defined in \eqref{eq:function phi}.
\end{theorem}
                                                                                                                                                                           
\begin{proof}
Denote by $A^I = (A^T)^{-1}$. By Corollary~\ref{col:dual frames}, $\mathcal{G}(h,\La(\Om),A\Z^m)$ is a dual system to $\mathcal{G}(g,\La(\Om),A\Z^m)$ if and only if
\begin{equation}\label{eq:function F}
F_\eta(t) := C(\eta) \sum_{l=1}^L \sum_{n\in\Z^m}\overline{\widehat{g_l}(t-An)}\, \widehat{h_l}(t-An- p_1^*(\eta)) = \delta_{\eta,0} \quad \mbox{for a.e. $t\in\R^m$}\,,
\end{equation}
where $C(\eta) = D(\La(\Om)) \,\Phi(-p_2^*(\eta))$. It is clear that for each $\eta\in\Ga^*$, $F_\eta \in L^1(A[0,1]^m)$ and is periodic, that is $F_\eta(t+Au)=F_\eta(t)$ for every $u\in\Z^m$. Therefore, we can compute the Fourier coefficients $\widehat{F_\eta}(v)$, $v\in\Z^m$, of $F_\eta(t)$ for each $\eta\in\Ga^*$:
\begin{align*}
\widehat{F_\eta}(v) &=C(\eta) \abs{\det A}^{-1} \int_{A[0,1]^m}  \sum_{l=1}^L \sum_{n\in\Z^m} \overline{\widehat{g_l}(t-An)}\, \widehat{h_l}(t-An- p_1^*(\eta)) \, e^{-2\pi i A^Iv \cdot t}\,dt\\
&= C(\eta) \abs{\det A}^{-1} \sum_{l=1}^L \int_{A[0,1]^m} \sum_{n\in\Z^m} \overline{\widehat{g_l}(t-An)}\, \widehat{h_l}(t-An- p_1^*(\eta)) \, e^{-2\pi i A^Iv \cdot t}\,dt\\ 
&= C(\eta) \abs{\det A}^{-1} \sum_{l=1}^L \int_{R^m} \overline{\widehat{g_l}(t)}\, \widehat{h_l}(t- p_1^*(\eta)) \, e^{-2\pi i A^Iv \cdot t}\,dt\\
&= C(\eta) \abs{\det A}^{-1} \sum_{l=1}^L \innerBig{T_{p_1^*(\eta)}\, \widehat{h_l}}{M_{A^I v} \,\widehat{g_l}}\,,
\end{align*}
where the interchanging of summation and integration is justified since, by Cauchy-Schwarz inequality, the sum $\sum_{n\in\Z^m}\overline{\widehat{g_l}(t-An)}\, \widehat{h_l}(t-An- p_1^*(\eta))$ is absolutely convergent. 
Therefore, by \eqref{eq:function F}, if $\eta\neq 0$, then $F_\eta(t) = 0$, and by the uniqueness of Fourier coefficients, we must have $\widehat{F_\eta}(v) = 0$ for every $v\in\Z^m$. For $\eta=0$, we have $F_\eta(t) = 1$, and therefore $\widehat{F_\eta}(v) = 0$ for $v\neq 0$ and $\widehat{F_\eta}(v) = 1$ for $v=0$. Since $ \inner{T_{p_1^*(\eta)}\, \widehat{h_l}}{M_{A^I v} \,\widehat{g_l}} = \inner{T_{A^I v}\,M_{p_1^*(\eta)} \,h_l}{g_l}$, summarizing above computation, \eqref{eq:function F}, is equivalent to
\begin{equation*}
C(\eta) \abs{\det A}^{-1} \sum_{l=1}^L \innerBig{T_{A^I v}\,M_{p_1^*(\eta)} \,h_l}{g_l} = \delta_{\eta,0}\, \delta_{v,0}\,,
\end{equation*}
for each  $\eta\in\Ga^*$ and $v\in\Z^m$, completing the proof.
\end{proof}

Theorem~\ref{thm:wexler raz} implies the necessary condition for a Gabor system $\mathcal{G}(g,\La(\Om), A\Z^m)$ to be a frame. A similar result, using different methods and for more general sets, was obtained in \cite{grrost17}.

\begin{corollary}
Let $g\in W(\Lt,\ell^1)$, $\La(\Om)$ a simple generic model set, $\Om$ symmetric around the origin, and let $A\in GL_m(\R)$. If  $\{ T_\la M_{An} g:\la\in\La(\Om), n\in\Z^m\}$ is a frame for $\Lt(\R^m)$ that has, for some $h\in W(\Lt,\ell^1)$, a  dual of the form $\{ T_\la M_{An} h:\la\in\La(\Om), n\in\Z^m\}$, then $D(\La(\Om)) \leq \abs{\det A}$.
\end{corollary}

\begin{proof}
Let $\{ T_\la M_{An} g:\la\in\La(\Om), n\in\Z^m\}$  be a frame for $\Lt(\R^m)$ with  dual $\{ T_\la M_{An} h:\la\in\La(\Om), n\in\Z^m\}$, for some $h\in W(\Lt,\ell^1)$. Let $B_g$ be the upper frame bound of $\{ T_\la M_{An} g:\la\in\La(\Om), n\in\Z^m\}$, and we can assume without loss of generality that $\norm{h}_2^2 = B_g^{-1}$. Then we have the frame decomposition
\begin{equation*}
\inner{f_1}{f_2} = \sum_{\la\in\La(\Om)} \sum_{n\in\Z^m} \inner{f_1}{T_\la M_{An} g} \inner{T_\la M_{An} h}{f_2} \quad \mbox{for all $f_1,f_2\in\Lt(\R^m)$,}
\end{equation*} 
If we set $f_1=h$ and $f_2=g$, by the Bessel property of $\{ T_\la M_{An} g:\la\in\La(\Om), n\in\Z^m\}$, we obtain
\begin{equation*}
\inner{h}{g} = \sum_{\la\in\La(\Om)} \sum_{n\in\Z^m} \abs{\inner{h}{T_\la M_{An} g}}^2 \leq B_g \norm{h}_2^2 = 1\,.
\end{equation*}
On the other hand, by Theorem~\ref{thm:wexler raz} with $L=1$, $\inner{h}{g} = D(\La(\Om))\, \abs{\det A}^{-1}$, and therefore $D(\La(\Om)) \leq \abs{\det A}$.
\end{proof}


\section{Conclusions}\label{sec:con}

We have characterized tight and  dual frames of translates on simple model sets, and we have applied the results to semi-regular Gabor systems. 
Our approach relied heavily on the existence of Poisson summation formula for model sets and the connection between the latter and almost periodic functions. The kind of results, to our knowledge, are the first ones in that direction. In the followup paper we address Gabor frames on $\La(\Om) \subset R^m\times \R^m$. They were recently investigated in \cite{K16}, and, on more general relatively separated sets, in \cite{GOR15}, however we will aim at the characterization in terms of equations.

\section{Acknowledgment}
This research was supported by the  Austrian Science Fund, FWF Project (V312-N25).


\section*{References}
\bibliographystyle{abbrv}

\end{document}